\documentclass[11pt]{article}
\usepackage{amssymb,amsfonts,amsmath,amsthm}
\usepackage[utf8]{inputenc}
\usepackage{caption}
\usepackage{subcaption}
\usepackage{authblk}
\usepackage{color}
\usepackage{esvect}
\usepackage{empheq} 
\usepackage{float}
\usepackage{graphicx}
\usepackage{afterpage}
\usepackage{mwe}
\usepackage{comment}
\usepackage{mathtools}
\usepackage{graphicx}
\def\rightangle{\vcenter{\hsize5.5pt
    \hbox to5.5pt{\vrule height7pt\hfill}
    \hrule}}

\numberwithin{equation}{section}
\allowdisplaybreaks
\newtheorem{thm}{Theorem}[section]
\newtheorem{defn}[thm]{Definition}
\newtheorem{cor}[thm]{Corollary}
\newtheorem{lemma}[thm]{Lemma}

\newtheorem{rmrk}[thm]{Remark}
\newtheorem{crl}[thm]{Corollary}
\newtheorem{exmp}[thm]{Example}

\newcommand{\R}{\mathbb{R}}

\newtheorem{prop}[thm]{Proposition}

\newcommand{\abs}[1]{\left\vert{#1}\right\vert}

\newcommand{\ba}{\begin{array}}
\newcommand{\ea}{\end{array}}

\newcommand{\bthm}{\begin{thm}}
\newcommand{\ethm}{\end{thm}}
\newcommand{\bstp}{\begin{stp}}
\newcommand{\estp}{\end{stp}}
\newcommand{\blemma}{\begin{lemma}}
\newcommand{\elemma}{\end{lemma}}
\newcommand{\bprop}{\begin{prop}}
\newcommand{\eprop}{\end{prop}}
\newcommand{\bpf}{\begin{pf}}
\newcommand{\epf}{\end{pf}}
\newcommand{\bdefn}{\begin{defn}}
\newcommand{\dH}{d\mathcal{H}}
\newcommand{\edefn}{\end{defn}}
\newcommand{\brk}{\begin{rmrk}}
\newcommand{\erk}{\end{rmrk}}
\newcommand{\bcrl}{\begin{crl}}
\newcommand{\ecrl}{\end{crl}}
\newcommand{\beg}{\begin{exmp}}
\newcommand{\eeg}{\end{exmp}}
\newcommand{\norm}[1]{\left\|#1\right\|}

\newcommand{\beqn}{\begin{equation}}
\newcommand{\eeqn}{\end{equation}}

\renewcommand{\leq}{\leqslant}
\renewcommand{\geq}{\geqslant}

\newcommand{\lm}{\lambda}

\newcommand{\mH}{\mathcal{H}}

\newcommand{\rec}{\mathcal{R}}

\newcommand{\beq}{\begin{equation}}
\newcommand{\eeq}{\end{equation}}
\newcommand{\bea}{\begin{eqnarray}}

\newcommand{\eea}{\end{eqnarray}}

\newcommand{\sh}{\mathcal{H}}

\newcommand{\sA}{\mathcal{A}}

\def\({\left(}
\def\){\right)}

\title{}
\author
\date

\begin{document}

\Large
 \begin{center}
\Large{Allen-Cahn Solutions with Triple Junction Structure at Infinity}\\ 

\hspace{10pt}

% Author names and affiliations
\large
\'Etienne Sandier$^1$, Peter Sternberg$^2$ \\

\hspace{10pt}

\small  
$^1$) Universit\'e Paris-Est Cr\'eteil \\
sandier@u-pec.fr
\\
$^2$) Department of Mathematics, Indiana University, Bloomington \\
sternber@iu.edu

\end{center}

\hspace{10pt}

\normalsize

%\address[label1]{Department of Mathematics}
%\address[label2]{Department of Mathematics, Indiana University, Rawles Hall
%831 East 3rd St., Bloomington, IN 47405}
\begin{abstract}
We construct an entire solution $U:\R^2\to\R^2$ to the elliptic system
\[
\Delta U=\nabla_uW(U),
\]
where $W:\R^2\to [0,\infty)$ is a 
`triple-well' potential. This solution is a local minimizer of the associated energy
\[
\int \frac{1}{2}\abs{\nabla U}^2+W(U)\,dx
\]
in the sense that $U$ minimizes the energy on any compact set among competitors agreeing with $U$ outside that set. Furthermore, we show that along subsequences, the `blowdowns' of $U$ given by $U_R(x):=U(Rx)$ approach a minimal triple junction as $R\to\infty$. Previous results had assumed various levels of symmetry for the potential and had not established local minimality, but here we make no such symmetry assumptions. 
\end{abstract}

\section{Introduction}
We will construct an entire solution $U:\R^2\to\R^2$ to the system
\begin{equation}
\Delta U=\nabla_uW(U),\label{PDE}
\end{equation}
which is minimizing on compact sets with respect to the associated energy
\[
E(u) =\int \frac{1}{2}\abs{\nabla u}^2+W(u)\,dx,
\]
where $W:\R^2\to [0,\infty)$ is a $C^2$
`triple-well' potential. That is, we assume that
\[\{p\in \R^2: W(p)=0\}=P:=\{p_1,p_2,p_3\},\]
and we assume non-degeneracy of the potential wells in the sense that
\begin{equation}
D^2 W(p_\ell)\geq b I\quad\mbox{for}\,\ell=1,2,3\;\mbox{for some}\;b>0,\;\mbox{where}\;I\;\mbox{is the}\;2\times 2\;\mbox{identity matrix.}\label{posdef}
\end{equation}

Additionally, we assume that for some $M>0$,
\begin{equation}
p\cdot \nabla W(p)\geq 0\quad\mbox{for}\;\abs{p}\geq M.
\label{Winfin}
\end{equation}

As in many studies of vector Allen-Cahn, we will make extensive use of the following degenerate Riemannian
metric on $\R^{2}$: 
\begin{equation}
d(p,q):=\inf\left\{\sqrt{2}\int_{0}^{1}W^{1/2}(\gamma(t))
\abs{\gamma'(t)}\;dt:\gamma\in C^{1}([0,1],\R^{2}),
\;\gamma(0)=p,\;\gamma(1)=q\right\},\label{Wdist}
\end{equation}
and we denote by $c_{ij}:=d(p_i,p_j)$ for $i\not=j$.
We will assume that the strict triangle inequality holds between the wells $p_1,p_2$ and $p_3$:
\begin{equation}
c_{12}<c_{13}+c_{23},\quad c_{13}<c_{12}+c_{23}\quad\mbox{and}\;c_{23}<c_{13}+c_{12}.\label{triangle}
\end{equation}
Under these assumptions, for $1\leq i<j\leq 3$ there exists at least one length-minimizing geodesic $\zeta_{ij}$ joining $p_i$ to $p_j$, see e.g. \cite{AF1,MS,SZun}. We will make the generic assumption that there is a {\it unique} such geodesic for each $i,j\in {1,2,3},\;i\not=j,$ though perhaps this can be relaxed. 

We also note that an equivalent variational description of the $c_{ij}$'s is given by
\beq
c_{ij}=\inf\left\{\int_{-\infty}^{\infty}W(f(t))+\frac{1}{2}
\abs{f'(t)}^2\;dt:\;f\in H^{1}_{loc}(\R,\R^{2}),
\;f(-\infty)=p_{i},\;f(\infty)=p_{j}\right\}.\label{other}
\eeq
Under an appropriate parametrization, we then find that each $\zeta_{ij}:\R\to\R^2$ satisfies the system
\beq
\zeta_{ij}''(t)=\nabla_uW(\zeta_{ij}(t))\quad\mbox{for}\;-\infty<t<\infty,\;\zeta_{ij}(-\infty)=p_i,\;\zeta_{ij}(\infty)=p_j.\label{hetero}
\eeq 
 From the perspective of ODE's, these geodesics $\zeta_{ij}$ represent heteroclinic connections between the potential wells.

We now denote by $\mathcal{A}$ the set of all functions $u^*:\R^2\to\R^2$ taking the form
 \beq
 u^*(x)=\left\{\begin{matrix} p_1&\;\mbox{on}\;S_1\\
 p_2&\mbox{on}\;S_2\\
 p_3&\mbox{on}\;S_3,\end{matrix}\right.\label{best}
 \eeq 
 where for $\ell=1,2$ and $3$, $S_\ell$ is a single (infinite) sector emanating from the origin with the three opening angles $\alpha_\ell$ given by
 \begin{equation}
\frac{\sin(\alpha_{1})}{c_{23}}=\frac{\sin(\alpha_{2})}
{c_{13}}=\frac{\sin(\alpha_{3})}{c_{12}}.\label{tj}
\end{equation}
See Figure \ref{triad}.

\begin{figure}[H]
	\centering
	\includegraphics[width = 0.6\textwidth]{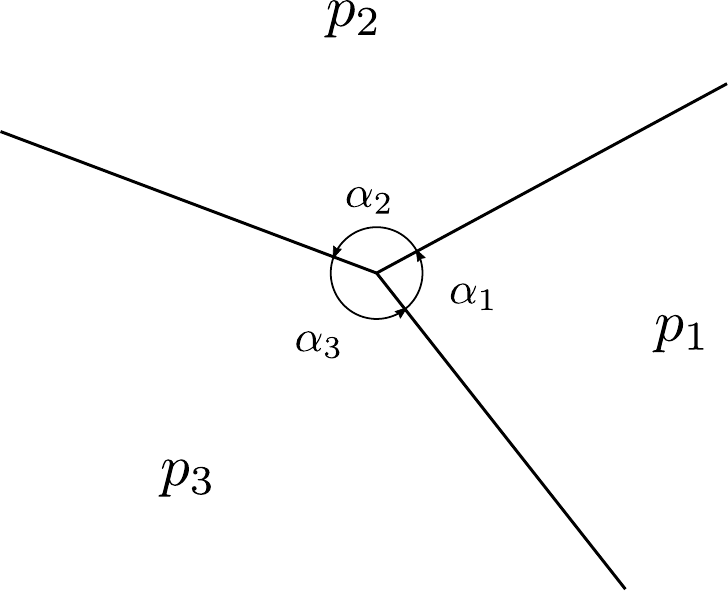}
	\caption{A locally minimizing partition of $\R^2$ with a triple junction.}
	\label{triad}
\end{figure}

The partition $\{S_1,S_2,S_3\}$ represents a locally minimizing partition of $\R^2$ with respect to the weighted perimeter functional 
 \beq
\{S_1,S_2,S_3\}\mapsto \sum_{1\leq i<j\leq 3}c_{ij}\mH^1\left(\partial S_i\cap \partial S_j\right),\label{fred}
 \eeq
 where $\mH^1$ refers to one-dimensional Hausdorff measure,
 and the condition \eqref{tj}  naturally arises as a criticality condition.
As we will recall in Section \ref{entire}, this partitioning problem represents the $\Gamma$-limit of a scaled version of the energy $E$, namely $E_R(u,\Omega)$, defined for any planar domain $\Omega$, any $R>0$ and any $u\in H^1(\Omega\;\R^2)$, via
\begin{equation}
E_R(u,\Omega)=\int_{\Omega}RW(u)
+\frac{1}{2R}\abs{\nabla u}^{2}\;dx.\label{ACenergy}
\end{equation}
We will write simply $E(u,\Omega)$ when referring to $E_1$ (i.e. $R=1$).
 
We will establish a connection between the structure at infinity of our entire solution $U$ to \eqref{PDE} and the triple junction partitions given by \eqref{best} by studying the asymptotic behavior of the blowdowns of $U$.

Our main result is the following:\\
\begin{thm} \label{main}There exists an entire solution $U:\R^2\to\R^2$ to 
\begin{equation}
\Delta U=\nabla_uW(U)\label{PDE1}
\end{equation}
which is a local minimizer of energy in the sense that 
for every compact set $K\subset\R^2$ and for every $v\in H^1_{loc}(\R^2;\R^2)$ satisfying $v=U$ on $\R^2\setminus K$ one has
\beq
E(U,K)\leq E(v,K).\label{italian}
\eeq
Furthermore, defining $U_R$ as the blowdown of $U$ via 
\beq
U_R(x):=U(Rx),\label{bd}
\eeq
we have that on any compact set $K\subset\R^2$:
\beq
{\rm{dist}}_{L^1(K;\R^2)}\left(U_R,\mathcal{A}\right)\to 0\;\mbox{as}\;R\to\infty.\label{wehope}
\eeq
That is,
\[
\lim_{R\to\infty}\left(\inf_{u^*\in\mathcal{A}}\norm{U_R-u^*}_{L^1(K;\R^2)}\right)= 0.
\]
\end{thm}
\begin{rmrk}
We believe that a stronger conclusion holds, namely that there exists a $u^*\in\mathcal{A}$ such that    
\[
\lim_{R\to\infty}\norm{U_R-u^*}_{L^1(K;\R^2)}= 0.
\]
\end{rmrk}

A step in the proof of the above theorem is the following result, of independent interest.

\begin{thm} \label{halfplanes} Assume  $U:\R^2\to\R^2$ is an entire solution to 
\eqref{PDE1} which is a local minimizer of energy such that for some sequence $R_j\to +\infty$, the sequence $U_{R_j}$ converges locally in $L^1$ to the function 
\beq\label{thehalf} u_0(x_1,x_2)=\left\{\begin{matrix} p_i&\;\text{if $x_2<0$.}\\
 p_j&\text{if $x_2>0$.}\end{matrix}\right.,\eeq
 for some pair $p_i\not=p_j$.
Then $U(x_1,x_2) = \zeta_{ij}(x_2+\Delta)$, for some $\Delta\in\R$.
\end{thm}

To place these results in context, we note that there is a large, and growing, collection of work on the general topic of finding entire solutions $u:\R^n\to\R^m$ to the vector Allen-Cahn system under various assumptions on the potential $W:\R^m\to\R$, on $n$ and on $m$. See, for example \cite{ABG,AF2,BFS2,BFS1,BGS,GS}. A source for a number of these results is the book \cite{AFS}. Most of these results, however, include some form of symmetry assumption on $W$. We also mention the recent work \cite{Bethuel} addressing concentration of general vector-valued critical points of Allen-Cahn in the plane.

Regarding the case under consideration here, namely $n=m=2$ and $W$ a triple well potential, an important
 first result on entire solutions appears in \cite{BGS}, where the authors assume the potential is equivariant by the symmetry group of the equilateral triangle. The convergence to the minimal triple junction partition \eqref{best}-\eqref{tj} they achieve under these symmetry assumptions (with necessarily each $\alpha_\ell=2\pi/3$) is much stronger than \eqref{wehope}. In particular, they show that
\[
\lim_{t\to\infty} U\left(t\frac{x}{\abs{x}}\right)=p_\ell\;\mbox{for}\;x\in S_\ell\;\mbox{off of the three rays}\;\partial S_k\cap\partial S_\ell\;\mbox{for}\;1\leq k<\ell\leq 3.
\]
On the other hand, since they work within the class of equivariant competitors, there is no claim of stability with respect to general perturbations. In a more recent contribution to this problem \cite{Fusco}, the symmetry assumptions on $W$ are weakened to include only the rotation subgroup of the full symmetry group of the equilateral triangle, thus relaxing the assumption of reflectional symmetry.

Our proof of Theorem \ref{main} proceeds by first appealing to \cite{SZi} to construct a sequence of $L^1$-local minimizers of $E_R$ on a particular nonconvex bounded domain, cf. Theorem \ref{Ziemer}. The candidate for our entire solution arises through a blow-up of this sequence, but care is needed here to execute the blow-up about a point where the local minimizers take a value far from the three heteroclinics.  This analysis is carried out in Section \ref{entire}, culminating in Proposition \ref{bigU}, where the blow-up limit $U$ is shown to be an entire, locally minimizing solution to \eqref{PDE} that avoids the three heteroclinics at the origin.

The next step involves an analysis of the blowdowns of any local minimizer $U$ in the sense of \eqref{italian}. Here we invoke the machinery of $\Gamma$-convergence, including an identification of the $\Gamma$-limit for vector Allen-Cahn subject to a Dirichlet condition carried out recently in \cite{Gaz}. We argue in Proposition \ref{compactness} that, up to passing to subsequences, these blowdowns converge to an $L^1$-local minimizer $u_0$ of the $\Gamma$-limit given in \eqref{Ezero}, which takes the form of a partitioning problem . 

The crucial estimate in our blowdown analysis comes in the form of an asymptotic equipartition of energy of any local minimizer, namely
\beq
\int_{B_{R}}\left(W(U)+\frac{1}{2}\abs{\nabla U}^2-\sqrt{2}\sqrt{W(U)}\abs{\nabla U}\right)\,dx<C_2R^{1-\alpha}\quad\mbox{for}\;R\gg 1,\label{eqpart}
\eeq
where $B_R$ is the disc of radius $R$ centered at the origin, and $C_2>0$ and $\alpha\in (0,1)$ are constants independent of $R$. This is established in Proposition \ref{equipartition}. The proof utilizes the regularity theory for the partitioning of a ball into three sets subject to a Dirichlet condition to obtain an upper bound on the energy, as well as a comparison between the infimum of such a partitioning problem and a related, less standard, partitioning problem described below  \eqref{introDir} to obtain a matching lower bound. We appeal to the regularity theory for both problems as recently presented in \cite{MNov}. We note that in \cite{Bethuel}
there appear other results on the asymptotic behavior of the `discrepancy measure,' that is, the integrand of \eqref{eqpart}, but these have a different nature given that they are derived only for critical points, not local minimizers, of the energy $E$.

From \eqref{eqpart} and a Pohozaev identity, we are able to establish the convergence of the blowdowns to a minimal cone via Lemmas \ref{limexist} and \ref{radrad} and Theorem \ref{RadialBD}. It is then simple to conclude that one of three limits must arise: either (i) the minimal cone is $\R^2$, i.e. $u_0=p_\ell$ for some $p_\ell\in P$, (ii) the minimal cone is a half-space, i.e. $u_0=p_i$ and $p_j$ for $i\not=j$ on either side of a line, or (iii) the minimal cone is given by three sectors satisfying \eqref{tj} so that $u_0$ is given by \eqref{best}, cf. Proposition \ref{trichotomy}.

Eliminating possibility (i) is easy, but eliminating (ii)--which roughly corresponds to arguing that at infinity, the entire solution $U$ does not look like a heteroclinic--is much more delicate. This is the content of Section \ref{halfplane}. The proof is by contradiction. We first obtain an upper bound for the energy that corresponds to the cost of a heteroclinic. Then we obtain a contradictory lower bound using crucially that $U$ was constructed in such a way that $U(0)$ is far from the three heteroclinics.

In Section \ref{lastpf}, we compare the two partitioning problems that emerge in our proof of \eqref{eqpart}. As one is somewhat non-standard, we hope this section will be of independent interest. The first problem involves the minimization of the partitioning functional
\beq
(S_1,S_2,S_3)\mapsto t_1\mH^1(\partial S_1\cap B)+t_2\mH^1(\partial S_2\cap B)+t_3\mH^1(\partial S_3\cap B),\label{joe}
\eeq
where $B\subset\R^2$ is a ball, $t_1,t_2$ and $t_3$ are positive numbers, and the admissible competitors $(S_1,S_2,S_3)$ are all partitions of $B$ satisfying a Dirichlet condition 
\beq
\partial S_\ell\cap\partial B=f^{-1}(p_\ell)\quad\mbox{for}\;\ell=1,2,3\quad\mbox{and}\;f\in BV(\partial B;P).\label{introDir}
\eeq
 (In the present context of Allen-Cahn, the coefficients $t_\ell$ are related to the constants $c_{ij}$ via \eqref{t123}, making \eqref{joe} equivalent to \eqref{fred}.) 
For the second problem, one fixes any number $\delta>0$ and then minimizes the same 
partitioning functional \eqref{joe} among triples $(S_1,S_2,S_3)$ of disjoint subsets of $B$ again subject to \eqref{introDir}, but now under the more relaxed condition that
$\abs{B\setminus \cup_\ell S_\ell}\leq \delta$. In other words, the competitors only need to ``almost partition" the ball. In Theorem \ref{almostpartitions}, we prove that the infimum of the second, more relaxed problem cannot lie more than $O(\delta^{1/2})$ below the infimum of the first problem.

In a personal communication in October of 2023, Nick Alikakos brought to our attention that he and Zhiyuan Geng were working on the same type of result. Their efforts eventually led to \cite{AZ2} and \cite{AZ}. They obtain the same conclusion as that of our Theorem \ref{main}, along with information about the proximity of the entire solution to the three potential wells along sequences of points going to infinity. The methods are quite different, with their result on convergence of blowdowns relying on a characterization of minimizing planar partitions into three sets, see Remark \ref{alternateproof}. As described above, our approach involves a new result on asymptotic equipartition of energy for local minimizers, along with the analysis of the rather novel geometry problem of ``almost partitions."

\section{Construction of a candidate for the entire solution}\label{entire}

Throughout this article, we will denote by $B_r(x)$ the ball in $\R^2$ of radius $r$
 and center $x$, unless the center is the origin, in which case we will simply write $B_r$.

\subsection{$\Gamma$-convergence results}
Our approach in this article will at times invoke $\Gamma$-convergence results relating the energy $E_R(u,\Omega)$ from \eqref{ACenergy} for a bounded domain $\Omega\subset\R^2$ to the functional
\beq
E_0(u,\Omega) := \sum_{1\leq i<j\leq 3}c_{ij}\sh^{1}(\partial^*S_{i}\cap\partial^*S_{j} \cap \Omega),\label{ez}
\eeq
where $S_{j}:=u^{-1}(p_{j})$ for $j=1,2,3$, and $\partial^*S$ refers to the reduced boundary of a set
$S$ of finite perimeter, cf. \cite{giusti}.

Building on previous $\Gamma$-convergence results for vector Modica-Mortola in the double-well case, e.g. \cite{FT, PSARMA, PSRocky}, the $\Gamma$-convergence of $\{E_R(\cdot,\Omega)\}$ to $E_0(\cdot,\Omega)$ for bounded $\Omega\subset\R^n$ in the setting of a multi-well potential and in the topology $L^1(\Omega;\R^n)$ is established in \cite{baldo}. 

We will also require a generalization of this $\Gamma$-convergence result to the situation where a Dirichlet condition is specified on $\partial \Omega$. Modica-Mortola type results that accommodate a Dirichlet condition appear in \cite{ORS} in the scalar setting and in \cite{Novack} in the context of the closely related Landau-deGennes energy. For our setting, however, we point to the recent result in \cite{Gaz}. For our purposes, it will suffice to state it for any bounded planar domain $\Omega$ with smooth boundary and for Dirichlet data taking values in the potential wells, though it holds more generally. To this end, let $h\in BV(\partial \Omega;P)$ and consider any sequence $\{h^R\}\subset H^1(\partial \Omega;\R^2)$ such that
\beq
  \abs{\partial_s h^R}\leq CR\;\mbox{for some}\;C>0\quad\mbox{and}\quad  h^R\to h\;\mbox{in}\; L^1(\partial \Omega;\R^2)\;\mbox{as}\; R\to\infty,\label{gamcon}
\eeq and such that
\beq
\int_{\partial \Omega} RW(h^R)
+\frac{1}{2R}\abs{\partial_s h^R}^{2}\;\dH^1<C\quad\mbox{for some constant}\;C\;\mbox{independent of}\;R.\label{DirData}
\eeq
Next define
\begin{align}
\tilde{E}_R(u,\Omega) := \left\{ 
\begin{array}{cc}
\displaystyle E_R(u,\Omega)
 &\mbox{if}\quad u\in H^1(\Omega;\R^2),\; u=h^R\;\mbox{on}\;\partial \Omega,\\ \\
+\infty & \mbox{ otherwise},
\end{array}
\right.
\end{align}
and define
\begin{align}
E^h_0(u,\Omega) := \left\{ 
\begin{array}{cc}
\displaystyle  \sqrt{2}\,E_0(u,\Omega)\,+\sum_{j=1}^3 \int_{S_j\cap\partial \Omega}d\left(p_j,h(x)\right)\,\dH^1& \mbox{if}\quad u\in BV(\Omega;P),\\ \\
+\infty & \mbox{ otherwise,}\label{Ezero}
\end{array}
\right.
\end{align}
where $E_0$ is defined in \eqref{ez} and $d(\cdot,\cdot)$ is given by \eqref{Wdist}.
Then we have \\
\noindent
\bthm\label{Gazoulis} (\cite{Gaz})
Assume $\{h^R\}$ satisfies \eqref{gamcon} and \eqref{DirData}. Then, as $R\to\infty$,
 the sequence
$\{\tilde{E}_R(\cdot,\Omega)\}$ has the $L^1-\Gamma$-limit $E^h_{0}(\cdot,\Omega)$. That is, for every $u\in L^1(\Omega;\R^2)$ we have the following two conditions:\\
(i) (Lower-semi-continuity) If $\{v_R\}\subset L^1(\Omega;\R^2)$ is any sequence converging to $u$ in $L^1$ then
\beq
\liminf_{R\to\infty}\tilde{E}_R(v_R,\Omega)\geq E^h_0(u),\label{lsc1}
\eeq
and\\
(ii) (Recovery sequence) There exists a sequence  $\{V_R\}\subset L^1(\Omega;\R^2)$  converging to $u$ in $L^1$ such that
\beq
\lim_{R\to\infty}\tilde{E}_R(V_R,\Omega)=E^h_0(u).\label{reco}
\eeq

\ethm

 \subsection{Construction of the entire solution via blow-up}
 
Our candidate for an entire solution satisfying Theorem \ref{main} will be constructed through a blow-up process, starting from an $L^1$-local minimizer of $E_R(\cdot,\Omega)$ for a particular choice of $\Omega$.
This local minimizer is, in turn, constructed in \cite{SZi} using $\Gamma$-convergence techniques. 

To place ourselves in the setting of \cite{SZi}, we fix any $u^*\in\mathcal{A}$ given by \eqref{best} and let $x_1,x_2$ and $x_3$ be the three points on $\partial B_1$ where the three phase boundaries  hit the unit circle. Then let $\Omega\subset \R^2$ be any bounded, simply connected open set containing $B_1$ such that $\partial\Omega$ is smooth and $\partial \Omega\cap\partial B_1=\{x_1,x_2,x_3\}$. Finally, assume that $\partial \Omega$ is strictly concave at these three points. See Figure \ref{Ziemerlocalmin}

Under these assumptions on $\Omega$, the following theorem is proven in \cite{SZi}, utilizing the local minimizer property associated with $\Gamma$-convergence established in \cite{KS}.

\begin{thm}\label{Ziemer}
 For $\Omega\subset\R^2$ as described above there exists a number $\delta_0>0$ such that for all $R$ sufficiently large, there exists an $L^1$-local minimizer $u_R$ of $E_R(\cdot,\Omega)$ in the sense that   
 \begin{equation}
E_R(u_R,\Omega)\leq E_R(v,\Omega)\quad\quad\mbox{provided}\;\norm{v-u_R}_{L^1(\Omega)}\leq\delta_0.\label{Rlocalmin}
\end{equation}
 Furthermore,
 \beq
u_R\to u^*\;\mbox{in}\; L^1(\Omega)\label{L1wp}
\eeq
and
\beq
E_R(u_R,\Omega)\to E_0(u^*,\Omega).\label{zlimit}
\eeq
Necessarily, such a local minimizer satisfies the Euler-Lagrange equation associated with $E_R$, namely
\begin{equation}
\frac{1}{R^2}\Delta u_R=\nabla_uW(u_R)\quad\mbox{in}\;\Omega,\label{vRPDE}
\end{equation}
along with homogeneous Neumann boundary conditions on $\partial\Omega$.
\end{thm}

\begin{figure}[H]
	\centering
	\includegraphics[width = 0.6\textwidth]{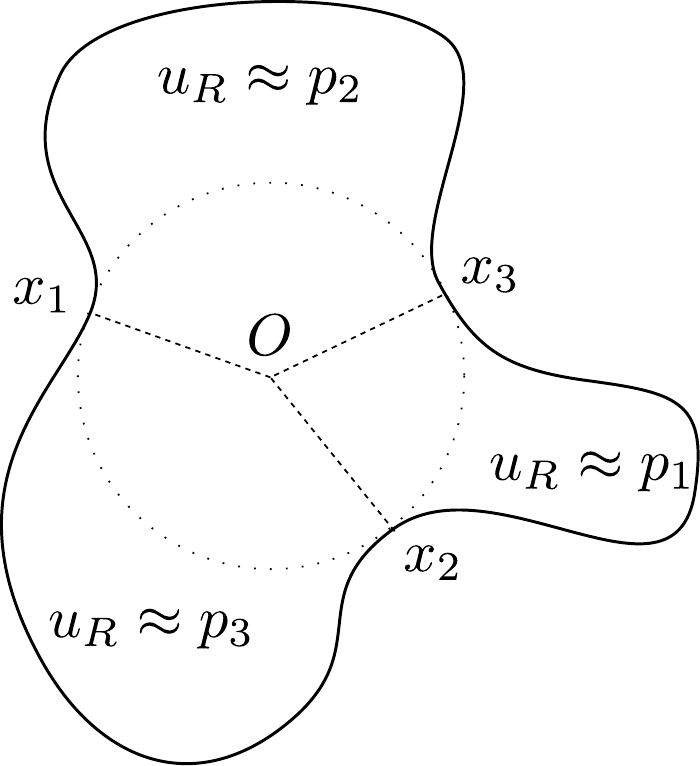}
	\caption{The $L^1$-local minimizer $u_R$.}
	\label{Ziemerlocalmin}
\end{figure}

Referring back to the three geodesics $\zeta_{ij}$ defined below \eqref{triangle}, we note that each is a simple curve (i.e. no self-intersections) and furthermore, any two of them, after including their endpoints, only intersect at one of their endpoints, e.g. $\zeta_{12}$ and $\zeta_{13}$ only intersect at $p_1$. This is because any transversal crossing would necessarily create a non-$C^1$ geodesic, violating regularity theory and any tangential intersection would violate the uniqueness of solutions to \eqref{hetero} subject to given initial conditions. As such, if we define $\Lambda$ as the union of the closure of the images of these three geodesics, that is, 
\beq
\Lambda:=P\cup \zeta_{12}(\R)\cup\zeta_{23}(\R)\cup\zeta_{13}(\R),\label{Lambda}
\eeq
then we can identify $\Lambda$ as a simple, closed curve in $\R^2$ passing through $p_1,p_2$ and $p_3$ which is smooth except at these 3 points. See Figure \ref{heteroclinics}.

\begin{figure}[H]
	\centering
	\includegraphics[width = 0.6\textwidth]{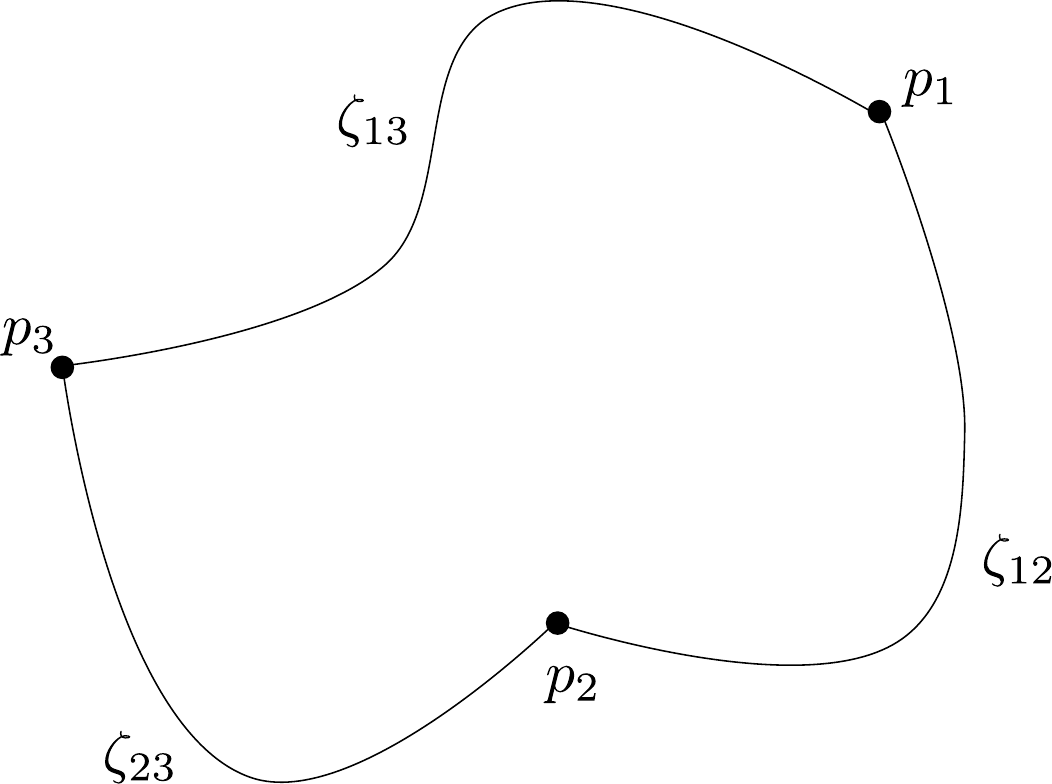}
	\caption{The closed curve $\Lambda$ consisting of the three heteroclinics.}
	\label{heteroclinics}
\end{figure}

An important property of the local minimizers constructed in Theorem \ref{Ziemer} is the following.

\begin{lemma}\label{centered} Let $\{u_R\}$ be the sequence of $L^1$ local minimizers established in Theorem \ref{Ziemer}. Then there exists a ball $B'$ compactly contained in $B_1$, a point $p$ inside $ \Lambda$ and a sequence of points $\{x_R\}\subset B'$ such that $u_R(x_R)=p$. In particular, there is a value $a_0>0$ such that
\[
{\rm {dist}}(u_R(x_R),\Lambda)>a_0.
\]
\end{lemma}
\begin{proof} We have that $\|u_R - u^*\|_{L^1(B_1)}$ tends to zero as $R\to +\infty$. Hence, 
by Fatou's Lemma,
\beq\label{l1fatou}0 \ge \int_0^1\liminf_{R\to +\infty}\|u_R - u^*\|_{L^1(\partial B_r)}\,dr\ge 0.\eeq
Similarly, for almost every $r\in(0,1)$, it holds that 
\beq\label{bdybound}\liminf_{R\to +\infty} E_R(u_R,\partial B_r) <+\infty.\eeq
It follows that  there exists $s\in (1/4,1/3)$ and $t\in(1/2, 2/3)$, and a subsequence still denoted $\{u_R\}$, such that $u_R\to u^*$ in $L^1(\partial \sA)$, where $\sA = B_t\setminus B_s$ and \eqref{bdybound} holds for $r = s,t$.

% Let us denote by $Y$ the phase boundary of the function $u^*$ given in \eqref{best}, that is
% \[
% Y:=\left(\partial \{u^*(x)=p_1\}\cup\partial\{u^*(x)=p_2\}\cup\partial\{u^*(x)=p_3\}\right)\cap B_1.
% \]
% We will first argue that on compact subsets of $B_1\setminus Y$, $u_R$ converges uniformly to one of the wells. To this end, fix any $\eta>0$ and then let $x_0\in B_1\cap\{u^*(x)=p_1\}$ be any point such that ${\rm{dist}}(x_0,Y\cup\partial \Omega)>2\eta.$ In view of \eqref{L1wp} and \eqref{zlimit}, it follows that we can find a radius $\rho\in (\eta,2\eta)$ such that $u_R\to p_1$ in $L^1(\partial B_\rho(x_0))$ and such that \eqref{DirData} holds with $\Omega$ replaced by $B_\rho(x_0)$ and $h^R$ replaced by the restriction of $u_R$ to $\partial B_\rho(x_0)$. 

Since it also follows from standard elliptic estimates that $\abs{\nabla u_R}\leq CR$, we have all the hypotheses of Theorem \ref{Gazoulis} satisfied on $\Omega = \sA$ and so we can assert the existence of a recovery sequence, say $\{\tilde{u}_R\}$, associated with $u^*$ and the boundary values of $u_R$ on $\partial \sA$. It then follows from the $L^1$-local minimality of $u_R$ \eqref{Rlocalmin} that
\begin{eqnarray*}
    &&(t-s)(c_{12}+ c_{23}+ c_{13})  \ge \lim_{R\to +\infty}\int_{s}^{t} E_R(\tilde u_R,\partial B_r)\,dr\\&& \ge \lim_{R\to +\infty}\int_{s}^{t} E_R(u_R,\partial B_r)\,dr \ge (t-s)(c_{12}+ c_{23}+ c_{13}).
    \end{eqnarray*}
Using \eqref{l1fatou} again, we deduce the existence of some $r\in(1/4, 2/3)$ such that $u_R\to u^*$ in $L^1(\partial B_r)$ and 
\beq\label{bdyengy}E_R(u_R,\partial B_r) \to c_{12}+ c_{23}+ c_{13}.\eeq
From the convergence of $\{u_R\}$ in $L^1(\partial B_r)$, we deduce the existence of three angles $\theta_i$, $i=1,2,3$, such that going to a further subsequence $u_R(re^{i\theta_i})$ converges to the well $p_i$ for each $i$. It then follows from \eqref{bdyengy} that $u_R$ is a minimizing sequence for the one dimensional energy $E_R$ on each of the arcs $A_{ij}$, where $A_{ij}$ is the portion of $\partial B_r$ between the angles $\theta_i$ and $\theta_j$.

In light of our assumption of uniqueness for the three heteroclinic connections, we may 
%pick any $r$, say, in the interval $(1/4,3/4)$ and 
assert that for $R$  sufficiently large one has
 \[
u_R(A_{ij})\;\mbox{is uniformly close to}\;\zeta_{ij}(\R)\;\mbox{for}\;1\leq i<j\leq 3.
\]
Consequently, the closed curve $u_R\left(\partial B_r\right)$ is uniformly close to the simple, closed curve $\Lambda$.

However, since $\Lambda$ is a Jordan curve, it partitions $\R^2$ into an inside, say $U$, and an outside unbounded set. For any $p\in U$, the index of $\Lambda$ with respect to $p$ is equal to $\pm 1$ and the same must be true for the curve $u_R\left(\partial B_r\right)$ since it is uniformly close to $\Lambda$ for $R$ large enough. Therefore, the latter curve cannot be homotopic to a constant in $\R^2\setminus \{p\}$, and so $p\in u_R\left(B(0,r)\right)$ for any $R$ large enough. Selecting  any $p\in U$ and any $x_R\in B(0,r)$ such that $u_R(x_R)=p$, the result follows.

\end{proof}

We now introduce our candidate for the entire solution of Theorem \ref{main} by taking a limit of blow-ups of $\{u_R\}$. 
\begin{prop}\label{bigU} Let $\{u_R\}$ be the sequence of local minimizers established in Theorem \ref{Ziemer}.
Let $\Omega^R:=\{x:\,\frac{x}{R}+x_R\in\Omega\}$ where $\{x_R\}$ is the sequence introduced in Lemma \ref{centered}. Also define $V^R(x):\Omega^R\to\R^2$  via $V^R(x):=u_R(\frac{x}{R}+x_R)$. Then there exists a subsequence $\{R_j\}\to \infty$ and a function $U:\R^2\to\R^2$ such that
\beq
V^{R_j}\to U \;\mbox{in}\;C^2\;\mbox{on compact subsets of}\;\R^2\label{blowuplimit}
\eeq
where $U$ solves \eqref{PDE}. Furthermore, $U$ is a local minimizer of $E$ in the sense of \eqref{italian}.
Finally, we have
\beq
{\rm {dist}}(U(0),\Lambda)>0.
%{\rm{inf}}\,\bigg\{\abs{U(0)-\zeta_{ij}(t)}:\,1\leq i<j\leq 3,\;-\infty<t<\infty\bigg\}\geq %a_0\quad\mbox{for some}\; a_0>0.
\label{origin}
\eeq
\end{prop}
\begin{rmrk}
For the remainder of the paper, when we say that a function $U$ is a local minimizer of $E$, we will always mean in the sense of \eqref{italian}.     
\end{rmrk}
\begin{proof}
Assumption \eqref{Winfin} implies through the maximum principle applied to $\abs{u_R}^2$ that $\norm{u_R}_{L^\infty(\Omega)}\leq M$ and so the same is true of $\{V_R\}$.  In light of \eqref{vRPDE} we observe that $V^R$ satisfies \eqref{PDE} on $\Omega^R$. Then standard elliptic estimates and bootstrapping leads, in particular, to uniform $C^{2,\alpha}$ bounds on compact sets for $\{V^R\}$. The conclusion \eqref{blowuplimit} follows as does the assertion that $U$ solves \eqref{PDE}.
 
 To establish the local minimality of $U$, fix any compact set $K$ and let $\tilde{v}:\R^2\to\R^2$ be any smooth function supported in $K$. Let $\tilde{v}_R(x):=\tilde{v}\big(R(x-x_R)\big)$ so that $\tilde{v}_R$ is supported in $x_R+\frac{1}{R}K$.  Then we have
 \[
 \int_K\abs{\tilde{v}_R}\,dx\leq \left(\max_K\abs{\tilde{v}}\right)\frac{1}{R^2}\abs{K}\quad\mbox{where}\;\abs{K}=\,\mbox{Lebesgue measure of}\;K.
 \]
 Now taking $R$ large enough so that $\left(\max_K\abs{\tilde{v}}\right)\frac{1}{R^2}\abs{K}<\frac{\delta_0}{2}$ we can invoke \eqref{Rlocalmin} to 
 conclude that
 \begin{align*}
 0&\leq  E_R(u_R+\tilde{v}_R,\Omega)-E_R(u_R,\Omega)\\ 
 &=E_R\big(u_R+\tilde{v}_R,x_R+\frac{1}{R}K\big)-E_R\big(u_R,x_R+\frac{1}{R}K\big) \\
 &=\frac{1}{R}\left(E(V^R+\tilde{v},K)-E(V^R,K)\right).
 \end{align*}
 Passing to the limit $R_j\to\infty$ in the inequality $ E(V^{R_j},K)\leq E(V^{R_j}+\tilde{v},K)$ we obtain \eqref{italian}.
 
 Property \eqref{origin} follows from Lemma \ref{centered} in light of the uniform convergence of $V^{R_j}\to U$.
\end{proof}

We conclude this section with a simple but crucial estimate on $U$ given in the following:
\blemma\label{ub} There exists a constant $C_1=C_1(W)$ such that for every $R>0$ one has
\beq
E(U,B_R)\leq C_1\,R.
\label{eub}
\eeq
\elemma

\begin{proof}
We may as well assume $R>1.$ Then we appeal to the local minimality of $U$, namely \eqref{italian}, with $v$ chosen to equal, say, $p_1$ on $B_{R-1}$ and then $v$ smoothly interpolating between $p_1$ and $U$ on the annulus $B_R\setminus B_{R-1}$. Since $E(v,B_{R-1})=0$ and $U$ and $\nabla U$ are uniformly bounded in terms of $W$ on the annulus, the result follows.
\end{proof}

\section{Blowdown Analysis}

In this section, we will characterize the limits of the blowdowns of any local minimizer of $E$. For this purpose we will need the following  compactness result associated with local minimizers of $E$. 
\bprop\label{compactness} Let $U:\R^2\to\R^2$ be a local minimizer of $E$. We have:\\
\noindent
(i) Let $\{R_j\}\to\infty$ be any sequence. Then there exists a subsequence $\{R_{j_k}\}$ and a function $u_0\in BV_{{\rm loc}}\left(\R^2;P\right)$ such that the blowdowns $\{U_{R_{j_k}} \}$ of $U$ satisfy
\beq
U_{R_{j_k}}\to u_0\quad\mbox{in}\; L^1_{{\rm loc}}(\R^2;\R^2).\label{Lone}
\eeq
(ii) After perhaps passing to a further subsequence (still denoted by $\{R_{j_k}\}$), one has for every $\ell\in\mathbb{Z}^+$ there exists a radius $\lm_\ell\in [\ell,\ell+1]$ such that 
\beq
\sup_k \left(\int_{\partial B_{\lm_\ell}} R_{j_k}W(U_{R_{j_k}})
+\frac{1}{2R_{j_k}}\abs{\nabla U_{R_{j_k}}}^{2}\;\dH^1\right)\leq 3C_1,\label{DirBd}
\eeq
where $C_1$ is the constant appearing in Lemma \ref{ub}.

Furthermore, 
\beq
U_{R_{j_k}}\to {\it trace\; of}\;u_0 \quad\mbox{in}\; L^1(\partial B_{\lm_\ell};\R^2).\label{limdir}
\eeq 
Lastly, $u_0$ is a local minimizer of $E_0(\cdot, \R^2)$ given in \eqref{ez} in the sense that
\beq
E_0(u_0,K)\leq E_0(v,K)\label{lmezero}
\eeq
for every compact $K\subset\R^2$ and every $v\in BV_{{\rm loc}}(\R^2;P)$ such that $v=u_0$ on $\R^2\setminus K.$
\eprop
\begin{proof}
Since $E(U,B_{\lm R})=R\,E_R(U_R,B_\lm)$, it then follows from Lemma \ref{ub} that
\beq
E_R(U_R,B_\lm)\leq \lm C_1\quad\mbox{for any}\;\lm>0.\label{uniebd}
\eeq 
Hence, the sequence $\{U_R\}$ has uniformly bounded energy on any ball $B_\lm$ and so the proof of (i) follows from \cite{baldo}, Prop. 4.1 using a diagonalization procedure.   To prove (ii), we note that from \eqref{uniebd}, in particular, it follows that
\beq
\int_1^2 \int_{\partial B_\lm} R_{j_k}W(U_{R_{j_k}})
+\frac{1}{2R_{j_k}}\abs{\nabla U_{R_{j_k}}}^{2}\;\dH^1\,d\lm<2C_1\quad\mbox{for all}\;k.\label{evid}
\eeq
Letting 
\[
f_k(\lm):=\int_{\partial B_\lm} R_{j_k}W(U_{R_{j_k}})
+\frac{1}{2R_{j_k}}\abs{\nabla U_{R_{j_k}}}^{2}\;\dH^1,
\]
let us suppose that \eqref{DirBd} is false for $\ell=1$. Then necessarily, for every $\lm \in [1,2]$ it would hold that
\beq\liminf_{k\to\infty} f_k(\lm)\geq 3C_1.\label{contdir}
\eeq
 Then,
by  Fatou's Lemma, 
\[ 2C_1\geq \liminf_{k\to\infty}\int_1^2f_k(\lm)\,d\lm\geq\int_1^2\liminf_{k\to\infty}f_k(\lm)\,d\lm\geq 3C_1,
\]
% \[
% 2C_1\geq \liminf_{k\to\infty}\int_1^2f_k(\lm)\,d\lm\geq\liminf_{k\to\infty}\int_1^2g_k(\lm)\,d\lm\geq \int_1^2 \lim_{k\to\infty} g_k(\lm)\,d\lm= M,
% \]
and a contradiction is reached. Passing to a further subsequence, the existence of a function $h\in BV(\partial B_{\lm_1};P)$ such that
$U_{R_{j_k}}\to h \quad\mbox{in}\; L^1(\partial B_{\lm_1};\R^2)$ follows from \eqref{DirBd} using the same compactness argument from \cite{baldo}, applied now to the energy restricted to the circle $\partial B_{\lm_1}$ with the full gradient replaced by the tangential gradient. 

 To establish \eqref{limdir} and the local minimality of $u_0$ we observe from the $\Gamma$-convergence result Theorem \ref{Gazoulis} that $u_0$ is the limit of minimizers of $\tilde{E}_{R_{j_k}}(\cdot,B_{\lm_1})$. Hence, $u_0$ must necessarily minimize $E_0^h(\cdot,B_{\lm_1})$. Indeed, for any $v\in BV(B_{\lm_1};P)$ one has
\[
E_0^h(v,B_\lm)=\lim_{k\to\infty} \tilde{E}_{R_{j_k}}(V_k,B_{\lm_1})\geq\liminf_{k\to\infty}\tilde{E}_{R_{j_k}}(U_{R_{j_k}},B_{\lm_1})\geq E_0^h(u_0,B_{\lm_1}),
\]
where $\{V_k\}$ is the recovery sequence associated with $v$ guaranteed to exist by Theorem \ref{Gazoulis}. We note that the first inequality above follows from the local minimality of $U_{R_{j_k}}$ in the sense of \eqref{italian}, since by construction, $V_k=U_{R_{j_k}}$ on $\partial B_{\lm_1}.$

 It follows that in fact, $u_0=h$ on $\partial B_{\lm_1}$, since, for example, if $h=p_1$ along some arc $\gamma\subset\partial B_{\lm_1}$, while the trace of $u_0=p_2$ on $\gamma$, then one could produce a lower energy competitor $v$ for the energy $E_0^h(\cdot, B_{\lm_1})$ by setting $v=p_1$ inside the slice of $B_{\lm_1}$ bounded by $\gamma$ and the secant line $L$ connecting the endpoints of $\gamma$. Then
 \[
E_0^h(u_0,B_{\lm_1})-E_0^h(v,B_{\lm_1})=d(p_1,p_2)\left(\mathcal{H}^1(\gamma)-\mathcal{H}^1(L)\right)>0,
 \]
 in light of the strict convexity of $B_{\lm_1}$, thus contradicting the minimality of $u_0$. Necessarily then, $u_0$ is also a minimizer of $E_0$ among competitors agreeing with $u_0$ on $\partial B_{\lm_1}$.
 
 We conclude the proof by noting that the same logic allows us to select a value $\lm_\ell\in [\ell,\ell+1]$ for every $\ell\in\mathbb{Z}^+$ and thus to conclude that $u_0$ minimizes $E_0$ in every ball $B_{\lm_{\ell}}$ among all competitors that agree with $u_0$ on $\partial B_{\lm_{\ell}}.$ Hence, \eqref{lmezero} holds. 
 \end{proof}
\subsection{Pohozaev and asymptotic equipartition of energy}

With an eye towards utilizing a Pohozaev identity, we next introduce the stress-energy tensor associated with a solution $U:\R^2\to\R^2$ to \eqref{PDE}:
\[
T_{ij}=U_{x_i}U_{x_j}-\delta_{ij}\left(\frac{1}{2}\abs{\nabla U}^2+W(U)\right).
\]
A standard calculation yields that $T$ is divergence-free. From this fact we get the Pohozaev identity on the ball $B_R$:
\[
\int_{B_R}\left(x_iT_{ij}\right)_{x_j}\,dx=\int_{B_R}\delta_{ij}T_{ij}+x_i(T_{ij})_{x_j}\,dx=\int_{B_R}{\rm tr}\,T\,dx.
\]
Applying the divergence theorem leads to
\[
R\int_{\partial B_R}\nu_iT_{ij}\nu_j\,\dH^1=-2\int_{B_R}W(U)\,dx,
\]
where $\nu=x/R$ is the outer unit normal to $B_R$. Using the definition of $T$ this can be written as
\beq
\frac{1}{2}\int_{\partial B_R} \frac{1}{2}\abs{U_\nu}^2-\frac{1}{2}\abs{U_s}^2-W(U)\,\dH^1=-\frac{1}{R}\int_{B_R}W(U)\,dx,\label{poho}
\eeq
where $U_s$ denotes the tangential derivative of $U$ along $\partial B_R$. Through \eqref{poho} we immediately obtain the following identity.
\bprop \label{Pohozaev}
Any entire solution $U$ to \eqref{PDE} satisfies
\beq\label{Poho1}
\frac{d}{dR}\left(\frac{1}{R}\int_{B_R}W(U)\,dx\right)=\frac{1}{2R}\int_{\partial B_R} \frac{1}{2}\abs{U_\nu}^2-\frac{1}{2}\abs{U_s}^2+W(U)\,\dH^1
\eeq
for all $R>0$.
\eprop

Our aim is to obtain a kind of asymptotic monotonicity result. To this end,
we define 
\beq\label{renorm}
\tilde{W}(R):=\frac{1}{R}\int_{B_R} W(U)\,dx.
\eeq
Then for any two values $0<R_1<R_2$ we integrate \eqref{Poho1} to find that
\begin{eqnarray}
&&\tilde{W}(R_2)-\tilde{W}(R_1)
=\int_{R_1}^{R_2}\frac{1}{2r}\int_{\partial B_r} \frac{1}{2}\abs{U_\nu}^2-\frac{1}{2}\abs{U_s}^2+W(U)\,\dH^1\,dr\nonumber\\
&&=\int_{B_{R_2}\setminus B_{R_1}} \frac{1}{2\abs{x}}\left( \frac{1}{2}\abs{U_\nu}^2-\frac{1}{2}\abs{U_s}^2+W(U)\right)\,dx\nonumber\\
&&=\int_{B_{R_2}\setminus B_{R_1}} \frac{1}{2\abs{x}}\left( W(U)-\frac{1}{2}\abs{\nabla U}^2+\abs{U_\nu}^2\right)\,dx.
\label{a}
\end{eqnarray}
Hence,
\begin{eqnarray}
&&     \tilde{W}(R_2)-\tilde{W}(R_1)
\geq -\int_{B_{R_2}\setminus B_{R_1}} \frac{1}{2\abs{x}}\abs{W(U)-\frac{1}{2}\abs{\nabla U}^2}\,dx\nonumber\\
&&\geq
-\frac{1}{2R_1}\int_{B_{R_2}\setminus B_{R_1}}\left(\sqrt{W(U)}-\frac{1}{\sqrt{2}}\abs{\nabla U}\right)\left(\sqrt{W(U)}+\frac{1}{\sqrt{2}}\abs{\nabla U}\right)\,dx\nonumber\\
&&\geq -\frac{1}{2R_1}\left\{\int_{B_{R_2}\setminus B_{R_1}}\left(\sqrt{W(U)}-\frac{1}{\sqrt{2}}\abs{\nabla U}\right)^2\,dx\right\}^{1/2}
\left\{\int_{B_{R_2}\setminus B_{R_1}}\left(\sqrt{W(U)}+\frac{1}{\sqrt{2}}\abs{\nabla U}\right)^2\,dx\right\}^{1/2}.\nonumber\\
&&\label{b}
\end{eqnarray}
Now in light of \eqref{eub}, we have that
\begin{eqnarray}
\left\{\int_{B_{R_2}\setminus B_{R_1}}\left(\sqrt{W(U)}+\frac{1}{\sqrt{2}}\abs{\nabla U}\right)^2\,dx\right\}^{1/2}&&\leq
\sqrt{2}\left\{\int_{B_{R_2}\setminus B_{R_1}}\left(W(U)+\frac{1}{2}\abs{\nabla U}^2\right)\,dx\right\}^{1/2}\nonumber\\
&&\leq \sqrt{2C_1}R_2^{1/2},\label{c}
\end{eqnarray}
so that \eqref{b} implies
\beq
  \tilde{W}(R_2)-\tilde{W}(R_1)
\geq-\sqrt{\frac{C_1}{2}}\frac{R_2^{1/2}}{R_1}\left\{\int_{B_{R_2}\setminus B_{R_1}}\left(\sqrt{W(U)}-\frac{1}{\sqrt{2}}\abs{\nabla U}\right)^2\,dx\right\}^{1/2}.\label{d}
\eeq
Inequality \eqref{d} shows that we can achieve an asymptotic monotonicity-type formula provided we can establish a sufficiently sharp measure of equipartition of energy.

The key estimate we will show is:
\bprop\label{equipartition} 
There exist constants $C_2>0$ and $\alpha\in (0,1)$,  such that for any local minimizer $U$ of $E$ and any $R$ sufficiently large one has the estimate
\beq
\int_{B_{R}}\left(\sqrt{W(U)}-\frac{1}{\sqrt{2}}\abs{\nabla U}\right)^2\,dx<C_2R^{1-\alpha}.
\label{keyest}
\eeq
\eprop

Our proof of Proposition \ref{equipartition} will involve the construction of a recovery sequence with a quantitative error bound, corresponding to a minimizer of $E_0(\cdot,B)$ in a ball $B$ subject to a general Dirichlet condition $h\in BV(\partial B;P)$, that is, among partitions $\{S_1,S_2,S_3\}$ of $B$ satisfying
\beq
\partial S_\ell\cap\partial B=h^{-1}(p_\ell)\quad\mbox{for}\;\ell=1,2,3.
\label{dircond}
\eeq
For this upper bound construction we will require a rather complete characterization of minimizers of this partitioning problem, to be found in Theorem \ref{Novack} and Corollary \ref{tripbound}. Our proof of Proposition \ref{equipartition} will also require a sharp lower bound for the energy of a related but somewhat non-standard partitioning problem. To state it, we first observe that given a partition, say $\{S_1,S_2,S_3\}$ of a ball $B$, its cost as given by $E_0$
can be equivalently expressed as
\beq
E_0(S_1,S_2,S_3)=t_1\mH^1(\partial S_1\cap B)+t_2\mH^1(\partial S_2\cap B)+t_3\mH^1(\partial S_3\cap B),\label{tform}
\eeq
where the numbers $t_1,\,t_2$ and $t_3$ are the solution to the system
\[
t_1+t_2=c_{12},\;t_1+t_3=c_{13},\;t_2+t_3=c_{23}.
\]
Solving, we find 
\beq
t_1=\frac{1}{2}\left(c_{12}+c_{13}-c_{23}\right),\;t_2=\frac{1}{2}\left(c_{12}+c_{23}-c_{13}\right),\;t_3=\frac{1}{2}\left(c_{13}+c_{23}-c_{12}\right),\label{t123}
\eeq
and so we note that each $t_j$ is positive in light of the assumption \eqref{triangle}. 

Then for any $\delta>0$ and any $h\in BV(\partial B;P) $ we consider the minimization of $E_0$ as given by \eqref{tform} among all disjoint subsets $\{S_1,S_2,S_3\}$ of $B$ satisfying the Dirichlet condition \eqref{dircond}, along with the constraint 
\beq
    \abs{B\setminus \left(\cup_{\ell=1}^3 S_\ell\right)}\leq \delta.\label{deltapart}
\eeq
We will require a good bound from below for the infimum of $E_0$ subject to \eqref{deltapart} and Dirichlet data $h\in BV(\partial B;P)$ in terms of the infimum of $E_0$ subject to the same Dirichlet condition but for  actual partitions of $B$, that is, with $\delta=0$ in \eqref{deltapart}. This is presented in Theorem \ref{almostpartitions}.
\vskip.1in
\noindent
{\it Proof of Prop. \ref{equipartition}}.\\
The proof of \eqref{keyest} will follow by first establishing an upper bound of the form
\beq
\int_{B_1}RW(U_R)+\frac{1}{2R}\abs{\nabla U_R}^2\,dx\leq m_R+\frac{C}{R^{\alpha}}\quad\mbox{for some}\;C>0\;\mbox{and}\;\alpha\in (0,1),
\label{aup1}\eeq
where $m_R$, defined below in \eqref{minEzero}, represents the minimal value of the partitioning problem $E_0$ subject to a certain Dirichlet condition related to $U_R$.
Then we will utilize Theorem \ref{almostpartitions} to establish a matching lower bound of the form 
\beq
\sqrt{2}\int_{B_1}\sqrt{W(U_R)}\abs{\nabla U_R}\,dx\geq m_R-\frac{C'}{R^{\alpha}}\quad\mbox{for some}\;C'>0.\label{biglb1}
\eeq
If we rephrase the desired upper and lower bounds \eqref{aup1} and \eqref{biglb1} in terms of $U$ instead of its blowdowns,
then the upper bound we seek takes the form
\beq
\int_{B_R}W(U)+\frac{1}{2}\abs{\nabla U}^2\,dx\leq Rm_R+CR^{1-\alpha},\label{aup}
\eeq
and the lower bound we want takes the form
\beq
\sqrt{2}\int_{B_R}\sqrt{W(U)}\abs{\nabla U}\,dx\geq Rm_R-C'R^{1-\alpha}.\label{biglb}
\eeq
The desired inequality \eqref{keyest}  then follows by combining \eqref{aup} and \eqref{biglb}.
\vskip.1in
\noindent{\bf Upper bound construction.}
\vskip.1in
Since by Proposition \ref{bigU}, $U_R$ minimizes $E_R(\cdot,B_1)$ among competitors sharing its boundary values on $\partial B_1$, we can obtain the upper bound through a construction of a low-energy competitor. In essence, this is akin to the recovery sequence construction for vector Allen-Cahn with a multi-well potential, adapted to handle a Dirichlet condition, as in the recent work \cite{Gaz}. The difference is that here this must be made {\it quantitative} with an error that is $O(R^{-\alpha})$. However, unlike the general recovery sequence construction, here we only need to build it for an $E_0$-{\it minimizing} partition that yields the value $m_R$ in the problem \eqref{minEzero} defined below.
 
To begin the pursuit of an upper bound, we first note that by \eqref{eub} we have for any $R>0$:
% \[
% \int_{R}^{2R}\int_{\partial B_r}W(U)+\frac{1}{2}\abs{\nabla U}^2\,d\mH^1\,dr=
% \int_{B_{2R}\setminus B_{R}}W(U)+\frac{1}{2}\abs{\nabla U}^2\,dx
% \leq 2C_1R.
% \]
 \[
\int_{R}^{2R}E\left(U,\partial B_r\right)\,dr = E\left(U,B_{2R}\setminus B_R\right)\leq 2C_1R.
\]
Hence, by the Mean Value Theorem, there exists a value $R'\in (R,2R)$ such that
\beq
%\int_{\partial B_{R'}}W(U)+\frac{1}{2}\abs{\nabla U}^2\,d\mH^1<2C_1.\label{2c}
E\left(U,\partial B_{R'}\right)\leq 2C_1.\label{2c}
\eeq
If we can establish \eqref{keyest} for $R'$, then replacing $C_2$ by $2^{1-\alpha}C_2$, we will have established \eqref{keyest} for $R$ as well.
Thus, with no loss of generality, we may assume that $R$ satisfies \eqref{2c} as well. Phrasing this condition in terms of the blowdowns $\{U_{R}\}$, the assumed bound takes the form
\beq
E_R\left(U_R,\partial B_1\right)\leq 2C_1.\label{wlog}
\eeq
The upper bound estimate \eqref{aup1} will result from the construction of a low-energy competitor for the minimization of $E_R(\cdot,B_1)$ that agrees with the blown down minimizer $U_R$ on $\partial B_1$.

 From \eqref{wlog}, it follows that off of a small set on $\partial B_1$, the function $U_{R}$ must stay near one of the three wells $p_1,p_2$ or $ p_3.$ 
 We now use this fact to identify a partition of $\partial B_1$  into three sets.

We note that in light of the non-degeneracy assumption \eqref{posdef}, there exists a positive number $\beta$, depending only on $W$, such that
\begin{equation}
W\;\mbox{is strictly convex for}\;\abs{p-p_j}<\beta,\;j=1,2,3.\label{convexngd}
\end{equation}
and furthermore,
\begin{equation}
\frac{b}{2}\abs{p_\ell-q}^2\leq W(q)\leq 2b\abs{p_\ell-q}^2.\label{cbds}   
\end{equation}

 Then let us define the set 
 \beq
 A_R:=\left\{x\in \partial B_1:\,d(U_R(x),P)>\frac{\beta}{2}\right\}.\label{TR}
 \eeq
This set is necessarily a union of open arcs. If such an arc $I$ possesses a point $x$ such that $d(U_R(x),P)\geq \beta$, then since at the endpoints of $I$, necessarily $U_R$ is at metric distance $\frac{\beta}{2}$ from $P$, it must be the case that $E_R\left(U_R,I\right)\geq C,$
% \[
  %\int_{I}RW(U_R)+\frac{1}{2}\abs{\nabla U_R}^2\,ds\geq C,
% \]
 for some positive constant $C$ depending only on $W$. We define $T_R$ to be the union of all such arcs, and so in light of \eqref{wlog}, we can assert that $T_R$ consists of a {\it finite} union of arcs whose total number is bounded by a constant depending only on $W$. It then follows from $\int_{T_R}RW(U_R)\,d\mH^1\leq 2C_1$ that 
 \beq
\mH^1\left(T_R\right)\leq CR^{-1}.\label{smalleye}
 \eeq
 On $\partial B_1\setminus T_R$ we note that the metric distance from $U_R$ to $P$ is less than $\beta$. 

 \begin{center}
     \underline{Boundary layer construction on the annulus $B_1\setminus B_{1-\rho}$}
 \end{center}

 We begin with the construction of a boundary layer on $B_1\setminus B_{1-\rho}$, where $\rho$ will be determined later. However, we will insist that
 \begin{equation}
  \rho\geq \frac{1}{R}.\label{bigrho}   
 \end{equation}
The number of disjoint arcs in $T_R$ is bounded by a constant depending on $W$ only, hence the same is true for the complement of $T_R$. We split this complement into two sets; $S_R$ and the remainder, the set $S_R$ being the union of arcs having length less than $\lm$, where $\lm\le 1$ is another parameter to be determined later. Let us denote the arcs in the remainder by say $\{I_k\}$ for $k=1,2,\ldots, N_R$, where $N_R$ is bounded by a constant $N_0=N_0(W)$. Then each $I_k$ will be of length at least $\lm$ and can be naturally associated with one of the wells in the sense that $U_R$ remains within a metric distance of $\beta$ from that well throughout $I_k$. 
 We now can expand each $I_k$ to a slightly larger arc $\tilde{I}_k$,  absorbing arcs of $T_R$ and of $S_R$ in the process, so as to form a partition of $\partial B_{1}$, where in light of \eqref{smalleye}, we know that
 \beq
 \mH^1\left(\partial B_1\setminus \cup I_k\right)=\mH^1\left(\cup \tilde{I}_k\setminus \cup I_k\right)\leq C(R^{-1}+\lm). \label{smalltil}
 \eeq
 We note that there is some ambiguity in terms of the assignment of an element of $P$ to arcs comprising  
$\cup\tilde{I}_k\setminus \cup I_k$. That is, if say $I_k$ is associated with $p_1$ and an adjacent arc $I_{k+1}$ is associated with $p_2$, then one can either expand $I_k$ into the gap between them and assign the value $p_1$ to the resulting $\tilde{I}_k$ or expand $I_{k+1}$ into the gap and assign the value $p_2$ to the resulting $\tilde{I}_{k+1}$. As well shall see, due to the smallness of these gaps guaranteed by \eqref{smalltil}, it will not matter which choice we make here.

On $\partial B_{1-\rho}$ we define a function $V_R(1-\rho,\theta)$ as follows.
 If $e^{i\theta}\in \tilde{I}_k$ and $e^{i\theta}$ is at least a distance $\frac{1}{R}$ from the endpoints of $\tilde{I}_k$, we take $V_R(1-\rho,\theta)$ to be equal to whichever well is associated with $I_k$. On the rest of $\partial B_{1-\rho}$, we define $V_R(1-\rho,\theta)$ through linear interpolation in $\theta$, so that
 \beq
\abs{\frac{\partial V}{\partial \theta}(1-\rho,\theta)}\leq CR.\label{joshua}
 \eeq

Now we define $V_R$ in the annulus $\mathcal{A}_{1,1-\rho}:=B_1\setminus B_{1-\rho}$ taking $V_R(r,\theta)$ to linearly interpolate in $r$ between $U_R(1,\theta)$ and $V_R(1-\rho,\theta)$ for each $\theta$. We estimate the energy in this annulus as follows: 

We begin with the cost of interpolation from $U_R(1,\theta)$ to $V_R(1-\rho,\theta)$ for $e^{i\theta}\in \cup_{k=1}^{N_R} I_k$. In view of \eqref{convexngd} and the fact that $V_R(1-\rho,\theta)$ is a constant equal to one of the wells on each $I_k$, we can invoke the convexity of all terms in the energy to assert that for any $t\in (0,1)$ one has
\beq
W(V_R(1-t\rho,\theta))\leq (1-t)W(U_R(1,\theta))\;\mbox{and}\; \abs{\frac{\partial V_R}{\partial\theta}(1-t\rho,\theta)}^2\leq (1-t) \abs{\frac{\partial 
 U_R}{\partial\theta}(1,\theta)}^2.\label{etienne}
\eeq

Estimating the radial derivative, we find
\beq
\abs{\frac{\partial V_R}{\partial r}(r,\theta)}^2\leq \frac{\abs{U_R(1,\theta)-p_\ell}^2}{\rho^2}\leq C\frac{W(U_R(1,\theta))}{\rho^2}\quad\mbox{for some}\;\ell\in\{1,2,3\},\label{Peter}
\eeq
for any $r\in (1-\rho,1).$
Combining \eqref{wlog}, \eqref{etienne} and \eqref{Peter}, we integrate over that part of the annulus $\mathcal{A}_{1,1-\rho}$ corresponding to the set of arcs $\cup_{k=1}^{N_R}I_k$, say $A'_{1,1-\rho}$, to obtain
\beq
E_R\left(V_R,A'_{1,1-\rho}\right)\leq C\left(\rho
+\frac{1}{R^2\rho}\right).\label{sandier}\eeq
%\int_{A'_{1,1-\rho}}RW(V_R)+\frac{1}{2R}\abs{\nabla V_R}^2\,dx\leq C\left(\rho
%+\frac{1}{R^2\rho}\right).\label{sandier}\eeq

Now  we turn to an estimate of the energetic cost in that portion of the annulus, say $A''_{1,1-\rho}$, corresponding to arcs in the complement of $\cup_{k=1}^{N_R}I_k$. Estimating the tangential derivative in $A''_{1,1-\rho}$, we find in view of \eqref{joshua} that
\[
\abs{\frac{\partial V_R}{\partial \theta}(r,\theta)}^2\leq CR^2,
\]
and for the normal derivative we have
\[
\abs{\frac{\partial V_R}{\partial r}(r,\theta)}^2\leq C\frac{1}{\rho^2}.
\]
Invoking \eqref{smalltil},  for the potential term we can then estimate that
\[
\int_{A''_{1,1-\rho}}RW(V_R)\,dx\leq CR\abs{A''_{1,1-\rho}}\leq CR\rho\left(R^{-1}+\lm\right),
\] 
and so
\begin{eqnarray}
E_R\left(V_R,A''_{1,1-\rho}\right)&\leq&
%\int_{A''_{1,1-\rho}}RW(V_R)+\frac{1}{2R}\abs{\nabla V_R}^2\,dx&&\leq
C\rho\left(\frac1R+\lm\right)\left(R+\frac{1}{R\rho^2}\right)\nonumber\\&=&C\left(\rho+\frac{1}{R^2\rho}+\rho\lm R+\frac{\lm}{R\rho}\right).\label{dmitry}
\end{eqnarray}
Therefore, summing \eqref{sandier} and \eqref{dmitry}, we see that, since $\lambda \le 1$ and $R\ge 1$,  
\beq
E_R\left(V_R,\mathcal{A}_{1,1-\rho}\right)\leq 
%\int_{\mathcal{A}_{1,1-\rho}}W(V_R)+\frac{1}{2R}\abs{\nabla V_R}^2\,dx\leq
C\left(\rho+\frac{1}{R\rho}+\rho\lm R\right).\label{golovaty}
\eeq
\vskip.1in
\begin{center}
    \underline{Construction of the competitor in $B_{1-\rho}$}
\end{center}
\vskip.1in
Let us now define $V_R$ in the ball $B_{1-\rho}.$ 
For this, we will introduce the minimization of $E_0$ subject to the Dirichlet condition $h_R:\partial B_{1-\rho}\to P$ satisfying $h_R(x)=p_\ell$ if $x/\abs{x}\in \tilde{I}_k$, where $I_k$ is the arc associated with $p_\ell$ and $\tilde{I}_k$ is its expansion, as described above \eqref{smalltil}:
\beq
m_R:=\inf \left\{E_0(u,B_{1-\rho}):\;u\in BV(B_{1-\rho};P), \;u=h_R\;\mbox{on}\;\partial B_{1-\rho}\right\}.\label{minEzero}
\eeq
Let $u_0$ denote a minimizer for \eqref{minEzero}. By Theorem \ref{Novack} we know that $\{u_0 = p_\ell\}$ is a union of no more than $N_0$ convex open sets, which we refer to as chambers. We emphasize that the constant $N_0=N_0(W)$ is independent of $R$. Furthermore, each chamber is bounded by a finite number of line segments and at least one boundary arc from the collection $\{\tilde{I}_k\}$. Lastly, from Corollary \ref{tripbound}, the number of triple junctions in the configuration $u_0$ is bounded by a constant depending only on $N_0$, thus again a number independent of $R$. 

Because each chamber contains the convex hull of $(1-\rho)\tilde{I}_k$ for some $k$
and each $\tilde{I}_k$ has arclength at least $\lm$, the thickness of each chamber is bounded below by $C\lm^2$, where the thickness is defined as the minimal distance between any pair of parallel supporting planes for the chamber.  

The map $V_R$ is defined in each chamber, say $\Omega$, as follows: Consider any segment $I$ of $\partial\Omega\cap B_{1-\rho}$ having length at least $2\eta$, where $\eta$ is to be determined later, but where we require that
\begin{equation}
    \eta\geq\frac{1}{R}.\label{biggeta}
\end{equation}
We then consider a sub-segment $J\subset I$ of length smaller by $\eta>0$ on each side of $I$ and consider a rectangle in $\Omega$ with base $J$ and height $h>0$. It is clear that if $h>0$ is small enough, then these rectangles are disjoint and included in $\Omega$. We now quantify how large $h$ is allowed to be for this property to still hold. We will always assume the bound
  \beq\label{deltaeta} 2h \le \eta.\eeq
  
To this aim, assume $h\in(0,\eta/2]$ is the largest height for which it holds that the rectangles are mutually disjoint and included in $\Omega$. For this value of $h$, if it is different from $\eta/2$, either two rectangles make contact with each other, or one rectangle makes contact with $\partial B_{1-\rho}$. In any case, let $p$ be the projection of the contact point onto $J$. 
We denote by $a$ and $b$ the endpoints of $I$.  Necessarily, there exists a point $q\in\partial \Omega\setminus I$ such that $\abs{p-q}\leq 2h$. Moreover, choosing the horizontal axis to be the line $L$ through $a$ and $b$, since the point $p$ is at distance at least $\eta$ from both $a$ and $b$, and since $2h\le \eta$, we have that the first coordinate of $q$ is between those of $a$ and $b$.  

Bringing in from infinity a line parallel to $L$ from the half-plane containing $q$ until it first touches $\Omega$, we denote this first contact point by $q'$,  and we denote by $p'$ the orthogonal projection of $q'$ onto $L$. 
%We may assume, without loss of generality that $p$ lies between $a$ and $p'$. 
Since $q\in\partial \Omega$, it cannot lie inside the triangle formed by $a,b$ and $q'$. This implies that the segment $[pq]$ intersects either $[aq']$ or $[bq']$. We assume the former, which implies that $q$ is further away from $p$ than $\alpha$, the orthogonal projection of $p$ onto the line containing $a$ and $q'$. So we have
\beq
\abs{p-\alpha}\leq \abs{p-q}\leq 2h.\label{ee}
\eeq
See Figure \ref{abqp}.
\begin{figure}[H]
	\centering
	\includegraphics[width = 0.6\textwidth]{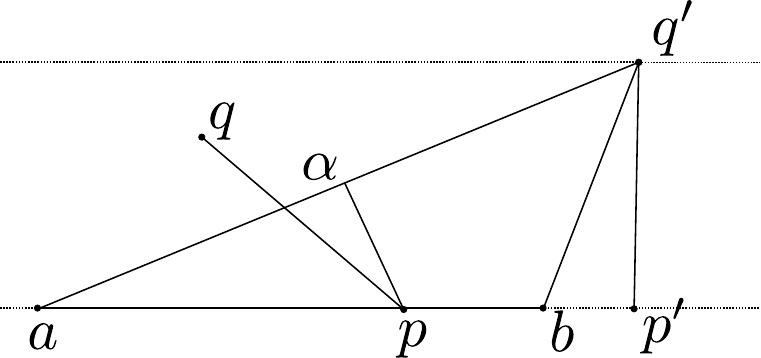}
	\caption{The configuration described above when two rectangles first touch.}
	\label{abqp}
\end{figure}

The triangles $a\alpha p$ and $aq'p'$ are similar. Therefore
  \beq
\frac{\abs{p-\alpha}}{\abs{a-p}}=\frac{\abs{p'-q'}}{\abs{a-q'}}.\label{ff}
  \eeq
Now $\abs{p'-q'}$ is at least the thickness of $\Omega$
so
\[
\abs{p'-q'}\geq C\lm^2.
\]
Furthermore, since $a$ and $q'$ lie in $B_{1-\rho}$ we know $\abs{a-q'}\leq 2$, and since $p\in J$, necessarily $\abs{p-a}\geq \eta$. Combining \eqref{ee}, \eqref{ff} with these inequalities, it follows that $h\geq C\eta\lm^2.$ Therefore, the rectangles will be disjoint and included in $\Omega$ as long as 
\beq
h < C\eta\lm^2,\label{deltacond}
\eeq
where this $C$ is $1/4$ of the constant $C$ appearing in the previous display.
 The case where the segment $[pq]$ intersects $[bq']$ also leads to \eqref{deltacond} in a similar manner.
  
Assuming \eqref{deltaeta}, \eqref{deltacond} are satisfied, consider a rectangle $\rec$ belonging to a chamber where $u_0 = p_i$ and sharing a boundary segment $J$ with a chamber where $u_0 = p_j$. Then, denoting by $s$ a coordinate orthogonal to $J$, we take $V_R=V_R(s)$ in $\rec$ given by $V_R(s) = \zeta_R(s):=\zeta_{ij}(Rs)$ for $0\leq s\leq h/2$, where  $\zeta_{ij}(0)$ is the midpoint of the heteroclinic $\zeta_{ij}$ with respect to the metric $d$. We take $\zeta_R$ to linearly interpolate between $\zeta_{ij}(Rh/2)$ and $p_i$ for $h/2\leq s\leq h$.

At this point we remark that from the assumption \eqref{posdef}, it follows that each $\zeta_{ij}$ approaches its end-states $p_i$ and $p_j$ at an exponential rate, i.e.
\beq
\abs{\zeta_{ij}(t)-p_j}\leq Ce^{-c(b)t}\quad\mbox{as} \;t\to\infty\label{expdecay}
\eeq
  for some constant $c(b)>0$, with a similar estimate holding as $t\to -\infty$.  Indeed, writing \eqref{hetero} as a first order autonomous system, say $z'=G(z)$, where
  \[z=(z_1,z_2,z_3,z_4)=\big(\zeta_{ij}^{(1)},\zeta_{ij}^{(2)},\zeta_{ij}^{(1)}\,',\zeta_{ij}^{(2)}\,'\big),
  \quad G(z)=\big(z_3,z_4,W_{z_1}(z_1,z_2),W_{z_1}(z_1,z_2)\big),\]
  one checks that at any $p_\ell\in P$, the $4\times 4$ matrix $DG(p_\ell)$ has eigenvalues $\pm\sqrt{\mu_1},\,\pm\sqrt{\mu_2}$ where $\mu_1,\mu_2\geq b>0$ are the eigenvalues of $D^2W(p_\ell)$. Thus, each $p_\ell$ represents a hyperbolic equilibrium point from the perspective of first order ODE theory and from local stable manifold theory the approach of $\zeta_{ij}$ to $p_i$ or $p_j$ as $t\to\pm\infty$ must be at an exponential rate as claimed in \eqref{expdecay}.

In light of \eqref{expdecay}, the modification can be made in such a way that 
\beq\label{enrec} E_R(V_R,\rec)\le \mH^1\left(\partial\rec\cap\partial\Omega\right) \(\frac12 c_{ij} + Ce^{-c(b) Rh/2}\).\eeq
In addition to \eqref{deltacond}, we will insist on a selection of $h$ such that
\beq
Rh\gg 1,\label{delbig}
\eeq
so that the exponential term in \eqref{enrec} will be negligible.

At this point, we consider an extension of $V_R(1,\theta)=U_R(1,\theta)$ to the annulus $\mathcal{A}_{1+\eta,1}:=B_{1+\eta}\setminus B_1$ that is constant along rays emanating from the origin. In light of \eqref{wlog}, we have
\beq\label{enann} E_R(V_R,\mathcal{A}_{1+\eta,1})\le C\eta.\eeq
Having defined $V_R$ in the rectangles of each chamber, we consider the finite collection of balls of radius $2\eta$
centered at  every one of the vertices of the polygonal curves $\partial\Omega\cap \overline{B}_{1-\rho}$ for every chamber $\Omega$. Any such vertex either coincides with the location of an endpoint of the arc $(1-\rho)\tilde{I}_k$ on $\partial B_{1-\rho}$, namely a point of discontinuity of $h_R$, or the location of an interior triple junction. Hence, the number of vertices, say $\tilde{N}_R$, is bounded by a  number $\bar{N}$ that is independent of $R$.  Referring to the collection of these balls as $\{B^j_{2\eta}\}_{j=1}^{\tilde{N}_R} $, we next define $V_R$ in the part of $B_{1-\rho}$ not belonging to any of the rectangles $\rec$ or $\cup_{j=1}^{\tilde{N}_R}B^j_{2\eta}$ by setting $V_R(x) = p_i$, when $x$ lies in a chamber associated with $p_i$. This is consistent with the boundary values on $\partial B_{1-\rho}$ since by construction $V_R(x)$ is equal to $p_i$ if $x$ is on the boundary arc $\partial B_{1-\rho}\cap \partial\Omega$ and at distance larger than $\eta$ from the ends of the arc, a condition satisfied when $x$ is not in any $B^j_{2\eta}$.

It remains to define $V_R$ inside the balls $B_{2\eta}^j$. To this end, we first note that for any ball $B^j_{2\eta}\subset B_{1-\rho}$, the boundary values of
 $V_R$ on $\partial B^j_{2\eta}$ vary between being constant or being given by a scaled heteroclinic, hence the tangential derivative of $V_R$ on this circle is bounded by $CR$. For any ball $B^j_{2\eta}$ not lying entirely in $B_{1-\rho}$, $V_R$ on $\partial B^j_{2\eta}$ maybe also be given partially by the linear interpolation construction carried out in the annulus $\mathcal{A}_{1,1-\rho}$ or, if the ball reaches $\mathcal{A}_{1+\eta,1}$, then it could partially coincide with the extension described above \eqref{enann}. However, in light of the assumption \eqref{bigrho}, in all cases the tangential derivative of $V_R$ along $\partial B^j_{2\eta}$ is bounded by $CR$ for some $C$ independent of $R$.

With this estimate in hand, we proceed to fill in the definition of $V_R$ in $\cup_{j=1}^{\tilde{N}_R}B^j_{2\eta}$ sequentially as follows. Starting with $B^1_{2\eta}$, we take $V_R$ to linearly interpolate between $V_R$, as previously defined, on $\partial B^1_{2\eta}$ and say, $p_1$ on $\partial B^j_{\eta}$. We then take $V_R\equiv p_1$ in the ball $B^j_{\eta}$. For such an interpolation, in light of assumption \eqref{biggeta}, we have
\begin{equation}
  \abs{\nabla V_R}\leq CR\;\mbox{in}\;B^1_{2\eta}\quad\mbox{and so}\quad
 E_R(V_R,B^1_{2\eta})\le CR\eta^2.  \label{enball}
\end{equation}
We then proceed to define $V_R$ in $B^2_{2\eta},\,B^3_{2\eta},\ldots $ in the same manner. However, it could happen that for some $i<j\in \{2,\ldots,N_R\}$ one has $B^j_{2\eta}\cap B^{i}_{2\eta}\not=\emptyset$. For example, this would occur if the centers of the two balls are vertices constituting endpoints of one side of a polygonal chamber of length less than $4\eta$. When such an intersection occurs, we simply define $V_R$ in the intersection of these two balls using the recipe for $V_R$ in $B^j_{2\eta}$. In light of \eqref{enball}, the $O(R)$ gradient bound is preserved through this process so that the $O(R\eta^2)$ energy bound is as well. Given that $N_R$ is uniformly bounded by $\bar{N}$, a constant independent of $R$, we can total the energy of $V_R$ inside $\cup_{j=1}^{\tilde{N}_R}B^j_{2\eta}$ to find an energetic contribution bounded by $CR\eta^2.$

\begin{center}
    \underline{Totaling the energetic cost of the construction}
\end{center}
Summing the bounds \eqref{enrec} and \eqref{enball} over rectangles and balls, we find that for any $h>0$ satisfying \eqref{deltacond} we have
\begin{eqnarray}  E_R(V_R,B_{1+\eta}) &\le& m_R \(1 + Ce^{-Rh/C}\)+CR\eta^2+ E_R(V_R,B_{1+\eta}\setminus B_{1-\rho})\nonumber\\
& \le& m_R + C\left(R\eta^2 + e^{-Rh/C} +\rho+\frac{1}{R\rho}+\rho\lm R+\eta\),\label{uplem}
\end{eqnarray}
where we have used \eqref{golovaty}, \eqref{enann}.

We may choose $\eta$, $\lm$, $\rho$, $h$ --- for instance setting $\eta = R^{-2/3}$, $\lm = R^{-1/8}$, $\rho = R^{-8/9}$ and $h = CR^{-11/12}$
--- so that for $R$ large enough the conditions\eqref{bigrho}, \eqref{biggeta}, \eqref{deltaeta}, \eqref{deltacond} and \eqref{delbig} are satisfied. For this choice, \eqref{uplem} implies that 
$E_R(V_R,B_{1+\eta}) \le  m_R  + CR^{-\alpha}$ with $\alpha = 1/8 - 1/9$. 

Finally, recalling that $V_R(1+\eta,\theta)=U_R(1,\theta)$ we must scale down this construction so that it agrees with $U_R$ on $\partial B_1$. Thus, with for example the choice $\eta=R^{-2/3}$ as above, we replace the sequence $V_R:B_{1+R^{-2/3}}\to \R^2$  with, say, $\bar{V}_R:B_1\to \R^2$ given by \[\bar{V}_R(r,\theta):=V_R\left((1+R^{-2/3})r,\theta\right).\]
Clearly such a scaling will only affect the energy bound by lower order terms, and since $E_R(U_R,B_1)\leq E_R(\bar{V}_R,B_1)$ we have established \eqref{aup1} for some $\alpha\in (0,1)$.  
\vskip.1in
\noindent
{\bf Matching lower bound}
\vskip.1in
We turn now to the task of obtaining a matching lower bound, namely  \eqref{biglb1}. 

Much in the same spirit as was done for the upper bound proof, we will replace the boundary values $U_R$ on $\partial B_1$ by much simpler boundary values through interpolation. This time, however, rather than interpolating from the boundary values $U_R$ to the much simpler boundary values $V(1-\rho,\theta)$ on $\partial B_{1-\rho}$ as we did in the argument leading up to \eqref{sandier}, we now define an {\it extension}, say $\tilde{U}_R$, of $U_R$ to a larger ball $B_{1+\rho}$ such that $\tilde{U}_R(1+\rho,\theta)=V_R(1-\rho,\theta).$
This amounts to reflecting across $\partial B_1$ the construction in the annulus $\mathcal{A}_{1,1-\rho}$ used in the upper bound argument to instead obtain an interpolation in the annulus
 $\mathcal{A}_{1+\rho,1}:=B_{1+\rho}\setminus B_1$.

We observe that precisely the same estimate \eqref{sandier} for the energetic cost of $V_R$ in the annulus $\mathcal{A}_{1,1-\rho}$  will now hold in the annulus $\mathcal{A}_{1+\rho,1}$ for the extension $\tilde{U}_R$. Hence, again making the choices taking $\rho=R^{-8/9}$ and $\lm=R^{-1/8}$ indicated below \eqref{uplem}, we conclude from \eqref{golovaty} that
\begin{align}
&\sqrt{2}\int_{B_1}\sqrt{W(U_R)}\abs{\nabla U_R}\,dx=\nonumber\\&\sqrt{2}\int_{B_{1+\rho}}\sqrt{W(\tilde{U}_R)}\abs{\nabla \tilde{U}_R}\,dx-
\sqrt{2}\int_{\mathcal{A}_{1+\rho,1}}\sqrt{W(\tilde{U}_R)}\abs{\nabla \tilde{U}_R}\,dx\nonumber\\
&
\geq \sqrt{2}\int_{B_{1+\rho}}\sqrt{W(\tilde{U}_R)}\abs{\nabla \tilde{U}_R}\,dx-E_R(\tilde{U}_R,\mathcal{A}_{1+\rho,1})\nonumber\\
&\geq
\sqrt{2}\int_{B_{1+\rho}}\sqrt{W(\tilde{U}_R)}\abs{\nabla \tilde{U}_R}\,dx-O(R^{-\alpha}),
\label{ext}   
\end{align}
where, as we did earlier, we have set $\alpha=1/8-1/9$.

We now define  three open subsets of $B_{1+\rho}$ via
 \[
\Omega_{\ell}^R:=\left\{x\in B_{1+\rho}:\,d(\tilde{U}_R(x),p_\ell)<t_\ell\right\}\;\mbox{for}\; \ell=1,2\;\mbox{and}\;3,
\]
where $t_1,$ $t_2$ and $t_3$ are defined in \eqref{t123}. 
These sets are disjoint since, for instance, $d(p,p_1)<t_1$ implies that 
$$d(p,p_2) > d(p_1,p_2) - t_1 = c_{12} - t_1 = t_2,$$
in light of \eqref{t123}.
Then we invoke the property of the metric $d$ that for any fixed
 `base point,' $p\in\R^2$, one has
\beq
\abs{\nabla_q d(p,q)}=\sqrt{W(q)}\;\mbox{for all}\;q\in\R^2,\label{nablad}
\eeq
cf. e.g. \cite{PSARMA}, along with 
 the co-area formula to estimate that
\begin{eqnarray}
&&\sqrt{2}\int_{B_{1+\rho}}\sqrt{W(\tilde{U}_R)}\abs{\nabla \tilde{U}_R}\,dx\geq \sum_{\ell=1}^3\sqrt{2}\int_{\Omega_{\ell}^R}\sqrt{W(\tilde{U}_R)}\abs{\nabla \tilde{U}_R}\,dx\nonumber\\
&&=\sum_{\ell=1}^3\int_{\Omega_{\ell}^R}\abs{\nabla d(\tilde{U}_R(x),p_\ell)}\,dx\geq\sum_{\ell=1}^3\int_{\frac{1}{R^{1/2}}}^{t_\ell}\mH^1\left(\big\{x:\,d(\tilde{U}_R(x),p_\ell)=s\big\}\right)\,ds
\nonumber\\
&&\geq \sum_{\ell=1}^3 \inf_{s\in [\frac{1}{R^{1/2}},t_\ell]}\mH^1\left(\big\{x:\,d(\tilde{U}_R(x),p_\ell)=s\big\}\right)\left(t_\ell-\frac{1}{R^{1/2}}\right)\nonumber\\
&&\geq  \sum_{\ell=1}^3 t_\ell \mH^1\left(\big\{x:\,d(\tilde{U}_R(x),p_\ell)=s_\ell^*\big\}\right)-\frac{C}{R^{1/2}},\label{lbj}
\end{eqnarray}
for some numbers $s_\ell^*\in [\frac{1}{R^{1/2}},t_\ell]$ for $\ell=1,2,3$ and some $C$ independent of $R$. (The lack of $R$ dependence for $C$ is clear since $C$ represents the minimal length of a level set  $\{d(\tilde{U}_R(x),p_\ell)=s\}$ among $s\in  [\frac{1}{R^{1/2}},t_\ell]$ within $\Omega_\ell^R.$)

With the goal of applying Theorem \ref{almostpartitions} to the triple 
$ \{x:\,d(\tilde{U}_R(x),p_\ell)<s_\ell^*\},\;\ell=1,2,3$,
 we now wish to estimate the measure of the set $B_1\setminus \cup_{\ell=1}^3 \{x:\,d(\tilde{U}_R(x),p_\ell)<s_\ell^*\}.$ Since $s_\ell^*\geq  \frac{1}{R^{1/2}}$ we have that
\beq
B_1\setminus \{x:\,d(\tilde{U}_R(x),P)<s_\ell^*\}\subset B_1 \setminus \{x:\,d(\tilde{U}_R(x),P)<\frac{1}{R^{1/2}}\}.\label{cl}
\eeq
But with an appeal to the nondegeneracy assumption \eqref{posdef} we can assert that if
\beq
d(\tilde{U}_R(x),p_\ell)\geq\frac{1}{R^{1/2}}\;\mbox{for}\;\ell=1,2\;\mbox{and}\;3,\;\mbox{then}\;W(\tilde{U}_R(x))\geq \frac{C}{R^{1/2}},\label{taylorb}
\eeq
for a constant $C$ depending on $W$.  Indeed, through an appeal to \eqref{cbds}, the definition of the metric $d$ and the convexity of $W$ near $P$, we see that for $\tilde{U}_R(x)$ in a $\beta$-neighborhood of $p_\ell\in P$, one has
\begin{align*}
 &   \frac{1}{R^{1/2}}\leq d(\tilde{U}_R(x),p_\ell) \leq 
\int_0^1\sqrt{W\left((1-t)p_\ell+t\tilde{U}_R(x)\right)}\abs{\tilde{U}_R(x)-p_\ell}\,dt\\
&\leq \int_0^1 \sqrt{tW(\tilde{U}_R(x))}\abs{\tilde{U}_R(x)-p_\ell}\,dt\leq
\frac{2}{3}\sqrt{2b}\abs{\tilde{U}_R(x)-p_\ell}^2.
\end{align*}
Then another appeal to \eqref{cbds} yields \eqref{taylorb}.

Then from the bound \eqref{eub}, it follows that
\[
C_1\geq\int_{B_1\setminus  \{x:\,d(\tilde{U}_R(x),P)<\frac{1}{R^{1/2}}\}}RW(\tilde{U}_R(x))\,dx\geq 
CR^{1/2}\abs{ B_1 \setminus \bigg\{x:\,d(\tilde{U}_R(x),P)<\frac{1}{R^{1/2}}\bigg\} }.
\]
Hence, in view of \eqref{cl}, we find that
\beq
\abs{B_1\setminus  \bigg\{x:\,d(\tilde{U}_R(x),P)<s_\ell^*\bigg\} }\leq \frac{C_0}{R^{1/2}}\label{admis}
\eeq
for some $C_0=C_0(W).$

With estimate \eqref{admis} in hand, we would like to now apply Theorem \ref{almostpartitions} to the triple $ \{x:\,d(\tilde{U}_R(x),p_\ell)<s_\ell^*\},\;\ell=1,2,3$ with the Dirichlet condition on $\partial B_{1+\rho}$ given by $h=h_R$, for $h_R$ as defined above \eqref{minEzero}.  However, to do so, we must make minor adjustments to these three sets near $\partial B_{1+\rho}$. These adjustments entail adding to or subtracting from these sets small slices of $B_{1+\rho}$ bounded by arcs of $\partial B_{1+\rho}$ and secant lines, so as to `fix' their traces to match those dictated by $h_R$. By our construction of $V(1-\rho,\theta)$, hence of $\tilde{U}_R(1+\rho,\theta)$, these adjustments occur along $N_R$ arcs that are contained in the set $\cup_{k=1}^{N_R} (1+\rho)\tilde{I}_k\setminus \cup (1+\rho)I_k$, where we recall that $N_R$ is bounded by a constant depending only on $W$. Therefore, in view of \eqref{smalltil}, one can alter the three sets so as to obtain a triple whose trace on $\partial B_{1+\rho}$ matches $h_R$ exactly, and the extra cost in perimeter will be $O(\lm)=O(R^{-1/8}).$
Furthermore, the estimate \eqref{admis} will still hold since the adjustments are lower order.

With this adjustment in hand, we now return to \eqref{ext} and \eqref{lbj}, and apply Theorem \ref{almostpartitions} with 
$\delta=\frac{C_0}{R^{1/2}}$, to obtain 
\begin{align}
&    \sqrt{2}\int_{B_1}\sqrt{W(U_R)}\abs{\nabla U_R}\,dx\nonumber \\&\geq 
    \sum_{\ell=1}^3 t_\ell \mH^1(\{x:\,d(\tilde{U}_R(x),p_\ell)=s_\ell^*\}
    -\frac{C}{R^{\alpha}}\nonumber\\
    &\geq\min\left\{E_0(u,B_{1+\rho}):\;u=h_R\;\mbox{on}\;\partial B_{1+\rho}\right\}-
\frac{\sqrt{C_0}\gamma(k)}{R^{1/4}}-\frac{C}{R^{1/2}}-\frac{C}{R^{1/8}}-\frac{C}{R^{\alpha}}\nonumber\\
&\geq
\min\left\{E_0(u,B_{1-\rho}):\;u=h_R\;\mbox{on}\;\partial B_{1+\rho}\right\}
   -
\frac{C}{R^{\alpha}}
=m_R-\frac{C}{R^{\alpha}}.
\end{align}
This is the lower bound \eqref{biglb1} we were seeking, and so the proof of Proposition \ref{equipartition} is complete.
\qed

\subsection{Convergence of the blowdowns to a minimal cone}
 With the crucial Proposition \ref{equipartition} now in hand, we apply \eqref{keyest} to \eqref{d} with $R_1=R$ and $R_2\in (R_1,2R_1]$ to obtain
 \beq
\tilde{W}(R_2)-\tilde{W}(R_1)
\geq- C_3R_1^{-\alpha/2}\label{e}
 \eeq
for $C_3$ depending only on $W$.

One consequence of \eqref{e} is:
\blemma\label{limexist} Assume $U:\R^2\to\R^2$ is a local minimizer of $E$ and define $\tilde W_R$ by \eqref{renorm}. Then the limit $L_0:=\lim_{R\to\infty}\tilde{W}(R)$ exists.
\elemma
\begin{proof}
Using \eqref{e} we will first argue that for any $\eta>0$ there exists a value $R_0>0$ such that 
\beq
\tilde{W}(R')-\tilde{W}(R)>-\eta\quad\mbox{whenever}\;R_0\leq R<R'.\label{claima}
\eeq
To see this, let $k$ be the largest integer such that $2^kR<R'.$ Then we see that
\begin{eqnarray*}
\tilde{W}(R')-\tilde{W}(R)&&=\tilde{W}(R')-\tilde{W}(2^kR)+\sum_{j=0}^{k-1}\left( \tilde{W}(2^{j+1}R)- \tilde{W}(2^{j}R)\right)\\
&&\geq - C_3\,\sum_{j=0}^{k}(2^{j}R)^{-\alpha/2 } \geq -C_3\frac{1}{R^{\alpha/2}}\left(\frac{1}{1-2^{\alpha/2}}\right).
\end{eqnarray*} 
Thus, taking $R$ large enough, we obtain \eqref{claima}.

Let us now suppose $\lim_{R\to\infty}\tilde{W}(R)$ does not exist and seek a contradiction. Since by \eqref{eub} we know that $0<\tilde{W}(R)\leq C_1$ for all $R>0$, this would imply that there exist sequences $\{R_j\}\to\infty$ and $\{R_k\}\to \infty$ such that $\lim_{R_j\to\infty}\tilde{W}(R_j)<\lim_{R_k\to\infty}\tilde{W}(R_k)$; say
\beq
\lim_{R_k\to\infty}\tilde{W}(R_k)-\lim_{R_j\to\infty}\tilde{W}(R_j)=L\quad\mbox{for some}\;L>0.\label{limcon}
\eeq
Hence, there exist $J$ and $K$ such that for $j>J$ and $k>K$ we would have
\[
\tilde{W}(R_k)-\tilde{W}(R_j)> \frac{L}{2}.
\]
However, by perhaps taking $j$ even larger we may find $R_j$ such that $R_j>R_k$ and then an application of \eqref{claima} with $\eta=\frac{L}{3}$ leads to the condition 
\[
\tilde{W}(R_j)-\tilde{W}(R_k)\geq -\frac{L}{3},
\]
and the contradiction is complete.
\end{proof}
Another consequence of \eqref{keyest} is the following:
\blemma\label{radrad} Assume $U:\R^2\to\R^2$ is a local minimizer of $E$. Then 
for every positive $\lm_1<\lm_2$ we have 
\beq
\lim_{R\to\infty}\int_{B_{ \lm_2 R}\setminus B_{\lm_1 R}}\frac{1}{\abs{x}}\abs{U_\nu}^2\,dx=0,\label{radiallim}
\eeq 
where $U_{\nu}=\nabla U\cdot \frac{x}{\abs{x}}.$

We also have
\beq
\lim_{R\to\infty}\int_{B_{\lm_2 R}\setminus B_{\lm_1R}} \frac{1}{2\abs{x}}\abs{W(U)-\frac{1}{2}\abs{\nabla U}^2}\,dx=0,\label{ep1}
\eeq
or equivalently,
\beq
\lim_{R\to\infty} \int_{B_{\lm_2 }\setminus B_{\lm_1}} \frac{1}{2\abs{x}}\abs{RW(U_R)-\frac{1}{2R}\abs{\nabla U_R}^2}\,dx=0.\label{ep2}
\eeq
\elemma
\begin{proof}
 To establish the limit \eqref{ep1}, we note that with the choices $R_1=\lm_1 R$ and $R_2=\lm_2 R$, the inequalities \eqref{b} and \eqref{c}, followed by application of Proposition \ref{equipartition} imply that
\begin{align*}
    &\int_{B_{\lm_2 R}\setminus B_{\lm_1R}} \frac{1}{2\abs{x}}\abs{W(U)-\frac{1}{2}\abs{\nabla U}^2}\,dx\\
    &\leq \sqrt{\frac{C_1\lm_2}{2\lm_1}}R^{-1/2}\left\{\int_{B_{\lm_2 R}\setminus B_{\lm_1 R}}\left(\sqrt{W(U)}-\frac{1}{\sqrt{2}}\abs{\nabla U}\right)^2\,dx\right\}^{1/2}\\
    &
    \leq \sqrt{\frac{C_1C_0\lm_2}{2\lm_1}}R^{-1/2}R^{1/2(1-\alpha)}=O(R^{-\alpha/2}).
\end{align*}

 Then \eqref{radiallim} follows from \eqref{a} and \eqref{ep1}, in light of Lemma \ref{limexist}.
\end{proof}
Now for any $\lm>0$ it follows from Lemma \ref{limexist} that
\[
\lim_{R\to\infty}\int_{B_{\lm}}RW(U_R)\,dx=\lm L_0.
\]
Combining this with \eqref{ep2} yields that for any $0<\lm_1<\lm_2$ one has 
\beq
\lim_{R\to\infty}E_R\left(U_R,B_{\lm_2}\setminus B_{\lm_1}\right)=2(\lm_2-\lm_1)L_0.
\eeq

We can rephrase \eqref{radiallim} in terms of the blowdowns as
\beq
\lim_{R\to\infty}\frac{1}{R}\int_{B_{ \lm_2 }\setminus B_{\lm_1 }}\frac{1}{\abs{x}}\abs{\nabla U_R\cdot \frac{x}{\abs{x}}}^2\,dx=0\quad\mbox{for any}\;0<\lm_1<\lm_2.\label{radiallim2}
\eeq
Now we will use this to argue that the limit of blowdowns is necessarily a cone:

\bthm\label{RadialBD} Assume $U:\R^2\to\R^2$ is a local minimizer of $E$. Let $\{R_j\}\to\infty$ be any sequence and let $\{R_{j_k}\}$ and a function $u_0\in BV(B_1;P)$ be any subsequence and subsequential $L^1$ limit guaranteed by Proposition \ref{compactness}. If we denote by $\Gamma_{i\ell}$ the phase boundary $\partial\{u_0=p_i\}\cap\partial\{u_0=p_{\ell}\}$, one has 
\beq
\nu_{i\ell}(x)\cdot x=0\quad\mbox{for every nonzero}\;x\in B_1\cap\Gamma^*_{i\ell}\;\mbox{and every}\;1\leq i<\ell\leq 3,\label{radbd} 
\eeq
where $\Gamma^*_{i\ell}$ denotes the reduced  boundary of $\Gamma_{i\ell}$ and $\nu_{i\ell}$ denotes a corresponding normal vector.
\ethm
\begin{proof}
We fix any positive number $\delta$ and note that for any $\mu\in (0,1)$ we have
\begin{eqnarray*}
\int_{B_1\setminus B_\mu}\sqrt{W(U_R)}\abs{\nabla U_R\cdot \frac{x}{\abs{x}}}\,dx&\leq&
\delta\,R\int_{B_1}W(U_R)\,dx+\frac{1}{\delta}\,\frac{1}{R}\int_{B_1\setminus B_{\mu }}\abs{\nabla U_R\cdot \frac{x}{\abs{x}}}^2\,dx\\
&\leq &\delta \tilde{W}(R)+\frac{1}{\delta}\,\frac{1}{R}\int_{B_1\setminus B_{\mu}}\frac{1}{\abs{x}}\abs{\nabla U_R\cdot \frac{x}{\abs{x}}}^2\,dx.
\end{eqnarray*}
Then sending $R\to\infty$ and invoking Lemma \ref{limexist} and \eqref{radiallim2} we conclude that
\[
\limsup_{R\to\infty}\int_{B_1\setminus B_\mu}\sqrt{W(U_R)}\abs{\nabla U_R\cdot \frac{x}{\abs{x}}}\,dx\leq \delta\,L_0,
\]
and since $\delta$ was arbitrary it follows that
\beq
\lim_{R\to \infty}\int_{B_1\setminus B_\mu}\sqrt{W(U_R)}\abs{\nabla U_R\cdot \frac{x}{\abs{x}}}\,dx=0.\label{ab}
\eeq
Next, we consider the function $x\mapsto d\left(p_1,U_R(x)\right)$. Suppressing subsequential notation, the fact that $U_{R_j}\to u_0$ in $L^1$ implies through definition \eqref{Wdist} that 
\[
d(p_1,U_{R_j})\stackrel{L^1(B_1)}{\longrightarrow} d(p_1,u_0)=\left\{\begin{matrix}0&\mbox{on}\;\{u_0=p_1\}\\
d(p_1,p_2)&\mbox{on}\;\{u_0=p_2\}\\
d(p_1,p_3)&\mbox{on}\;\{u_0=p_3\}\end{matrix}\right.,
\]
and we note that since $u_0\in BV(B_1;\R^2)$ one has $x\mapsto d(p_1,u_0(x))\in BV(B_1).$
Hence, in the sense of distributions, we have
\beq
\nabla_x d(p_1,U_{R_j}(x))\to \nabla_x d(p_1,u_0(x))=\sum_{1\leq i<\ell\leq 3}\left(d(p_1,p_\ell)-d(p_1,p_i)\right)\nu_{i\ell}\,\sh^1\rightangle \Gamma^*_{i\ell},\label{distrib}
\eeq
where this form of the distributional gradient follows, for example, from \cite{giusti}, Prop. 2.8.
Then recalling \eqref{nablad}, an elementary  chain rule calculation shows that
\[
\abs{\nabla_x d\left(p_1,U_R(x)\right))\cdot \frac{x}{\abs{x}}}\leq \sqrt{W\left(U_R(x)\right)}\abs{\nabla U_R\cdot \frac{x}{\abs{x}}},
\]
and so from \eqref{ab} we may conclude that
\beq
\lim_{R\to \infty}\int_{B_1\setminus B_\mu}\abs{\nabla_x d\left(p_1,U_R(x)\right))\cdot \frac{x}{\abs{x}}}\,dx=0.\label{ac}
\eeq
Fix now any non-zero point $x_0\in \Gamma^*_{i\ell}$ for some $1\leq i<\ell\leq 3$, and fix $\mu>0$ less than $\abs{x_0}$. Then we take an arbitary $\phi\in C_0^1\left(B_r(x_0)\right)$,
with $r$ chosen small enough so that $B_r(x_0)\subset B_1\setminus B_\mu.$ It follows from \eqref{distrib} and \eqref{ac} that
\[
\int_{\Gamma^*_{i\ell}\cap B_r(x_0)}\phi\,\nu_{i\ell}\cdot\frac{x}{\abs{x}}\,d\sh^1(x)=0.
\]
Since $\phi$ is arbitrary, we obtain the desired property \eqref{radbd}.
\end{proof}

In light of Proposition \ref{compactness}, we know that any limit of blowdowns, $u_0$, minimizes $E_0$ subject to its own boundary values. Now that we also know any limit of blowdowns is a cone, it follows immediately from Theorem \ref{Novack} that there are only three possibilities:
\bprop\label{trichotomy} Under the hypothesis and with the notations of Theorem~\ref{RadialBD},  either\\
\beq (i) \quad u_0\equiv p_i\;\mbox{for some}\; i\in\{1,2,3\},\label{con}
\eeq
or there exists a half-plane $H$ with $\partial H$ passing through the origin such that
\begin{equation}
(ii)\qquad u_0(x)=\left\{\begin{matrix}p_i&\;\mbox{in}\;H\\ p_\ell&\;\mbox{in}\;\R^2\setminus H,\end{matrix}\right.\label{half}
\end{equation}
for some $i,\ell\in\{1,2,3\}$ with $i\not=\ell$,
or
\beq
(iii)\qquad u_0\;\mbox{takes the form}\;\eqref{best}\label{tripop}
\eeq
with the three sectors $S_1,\,S_2$ and $S_3$ having opening angles $\alpha_1,\,\alpha_2$ and $\alpha_3$ satisfying the condition \eqref{tj}. 
\eprop
\begin{rmrk}\label{alternateproof}
 It was recently pointed out to us by Michael Novack that there is a different avenue available to reach the conclusion that any subsequential limit $u_0$ of blowdowns of a local minimizer of $E$ must satisfy either \eqref{con}, \eqref{half} and \eqref{tripop}. This alternative argument utilizes, among other tools, a monotonicity formula for minimizing partitions adapted to this setting, along with available regularity theory such as that described in Theorem \ref{Novack}, to show that the only locally minimizing partitions of the plane with respect to the energy $E_0$ are partitions fitting one of these three descriptions. Then coupled with Proposition \ref{compactness} we would reach the same conclusion as that of Proposition \ref{trichotomy}. See also \cite{Alikexp} for a presentation of this property of minimizing partitions.  
\end{rmrk}

\section{Eliminating the half-plane}\label{halfplane}
We will now argue that in our setting, neither \eqref{con} or \eqref{half} are possible, leaving \eqref{tripop} as the only option, thus leading to a proof of our main result, Theorem \ref{main}. 

We will first prove a `clearing-out' type of result, saying that sufficiently low energy in a ball implies uniform nearness to a potential well on a smaller ball. We remark that a result of this type for mere solutions to \eqref{vRPDE} is established in \cite{Bethuel}, Prop. 6.4, but since the proof is considerably simpler in the setting of local minimizers, we present a proof in this setting below.
%\bprop\label{Fabrice} (\cite{Bethuel}, Prop 6.4). For any $R\geq 1$, let $z_R$ be any solution %of \eqref{vRPDE}. Then there exist positive constants $\eta_1$ and $C^*$ depending on the %potential $W$ such that if
%\[
%E_R(z_R,B_1)\leq \eta_1,
%\]
%then for some $\ell\in\{1,2,3\}$ we have
%\beq
%\abs{z_R(x)-p_\ell}\leq C^*\left(E_R\left(z_R,B_1\right)\right)^{1/6}\quad\mbox{for all}\;x\in %B_{3/4}.
%\label{bethuel}
%\eeq
%\eprop
\bprop \label{newclearing} For any $R>0$, let $z_R$ be a local minimizer of $E_R$ satifying a gradient bound $\abs{\nabla z_R}\leq C_1R$ for some $C_1>0$. Then there exists a number $\eta$ depending only on $W$ such that if
\[
E_R(z_R,B_{r_0})< \eta\quad\mbox{on some ball}\;B_{r_0},
\]
then there exists a point $p_\ell\in P$ and a value $\Bar{R}>0$ such that for all $x\in B_{r_0/2}$, one has the uniform estimate
\[
\abs{z_R(x)-p_\ell}<\sqrt{3}\bigg(\frac{2}{b}\bigg)^{1/4}\bigg(E_R(z_R,B_{r_0})\bigg)^{1/2}
\]
for all $R\geq \Bar{R},$ where $b$ is the constant appearing in \eqref{cbds} and $\Bar{R}=\Bar{R}(b,r_0)$.
\eprop

\begin{proof}
With no loss of generality, we take $B_{r_0}$ to be centered at the origin.
For ease of notation, we will write
$e_R:=E_R(z_R,B_{r_0})$, so that our hypothesis is that $e_R<\eta$ with $\eta$ to be specified shortly.

We begin with the observation that if $q\in\R^2$ is any point such that $\abs{q-p_\ell}\leq \beta$ for some $p_\ell\in P$, then invoking \eqref{Wdist} and \eqref{cbds}, we have
\begin{align*}
 \sqrt{\frac{b}{2}}&\min_{\gamma(0)=p_\ell,\gamma(1)=q}\int_0^1\abs{\gamma(t)-p_\ell}\abs{\gamma'(t)}\,dt\leq  d(q,p_\ell)\leq\\
&\sqrt{b}\min_{\gamma(0)=p_\ell,\gamma(1)=q}\int_0^1\abs{\gamma(t)-p_\ell}\abs{\gamma'(t)}\,dt .
 \end{align*}
 Hence, 
 \beq
 \sqrt{\frac{b}{2}}\abs{q-p_\ell}^2\leq d(q,p_\ell)\leq \sqrt{b}\abs{q-p_\ell}^2.\label{compare}
\eeq

Applying the Mean Value Theorem, the assumption $e_R< \eta$ on $B_{r_0}$ allows us to find a radius, say $r^*\in (r_0/2,r_0)$, such that
\beq
E_R(z_R,\partial B_{r^*})<2e_R<2\eta.\label{twoeta}
\eeq
Consequently, for any two points $x_1,x_2\in\partial B_{r^*}$ one has
\[
d\left(z_R(x_1),z_R(x_2)\right)\leq \sqrt{2}\int_{\partial B_{r^*}}\sqrt{W(z_R)}\abs{\nabla z_R}\,ds<2e_R.
\]
Then the fact that $\int_{\partial B_{r^*}}RW(z_R)\,ds<2e_R$ implies that for $R$ large enough there exists a point, say $x_R\in
\partial B_{r^*},$ such that $d(z_R(x_R),p_\ell)<e_R$ for some $p_\ell\in P.$ Consequently, it follows from the triangle inequality that 
\begin{equation}
 d(z_R(x),p_\ell)<3e_R\quad\mbox{for all}\;x\in \partial B_{r^*}. \label{threeeta}
\end{equation}
We now impose the condition
\beq
\eta\leq\sqrt{\frac{b}{2}}\,\frac{\beta^2}{3}.\label{firsteta}
\eeq
It follows that 
\begin{equation}
\abs{z_R(x)-p_\ell}\leq\beta\quad\mbox{for all}\; x\in\partial B_{r^*}.\label{unizr}
\end{equation}
Otherwise, for some $x\in\partial B_{r^*}$, we would have
\[
d(z_R(x),p_\ell)\geq\min_{\{q:\,\abs{q-p_\ell}=\beta\}}d(q,p_\ell)\geq\sqrt{\frac{b}{2}}\beta^2, 
\]
contradicting \eqref{threeeta}, given that $e_R<\eta$.

Having established \eqref{unizr}, we now appeal to the local minimality of $z_R$ by constructing a competitor, say $v_R$, in $B_{r^*}$ that linearly interpolates on the annulus $\mathcal{A}_{r^*,r^*-\frac{1}{R}}$ between $z_R(x)$ on $\partial B_{r^*}$ and $p_\ell$ on $\partial B_{r^*-\frac{1}{R}}$
via the formula
\beq
v_R(x):=\lm_R(\abs{x})z_R\left(r^*\frac{x}{\abs{x}}\right)+\big(1-\lm_R(\abs{x})\big)p_\ell,
\label{vsubRdefn}
\eeq
for $r^*-\frac{1}{R}\leq \abs{x}\leq r^*$, where $\lm_R(r):=R(r-r^*)+1$.
We compute that
\beq
\nabla v_R(x)=\lm_R(\abs{x})\nabla  z_R\left(r^*\frac{x}{\abs{x}}\right)+R\frac{x}{\abs{x}}\otimes 
\left(z_R\left(r^*\frac{x}{\abs{x}}\right) -p_\ell\right),
\eeq
so that 
\beq
\abs{\nabla v_R(x)}^2\leq 2\abs{\nabla  z_R\left(r^*\frac{x}{\abs{x}}\right)}^2+2R^2\abs{z_R\left(r^*\frac{x}{\abs{x}}\right) -p_\ell}^2.\label{vrgrad}
\eeq

Property \eqref{unizr} guarantees that this interpolation always yields values inside the ball of radius $\beta$ about $p_\ell$ so that $v_R$ takes its values in the region where $W$ is convex.
This convexity allows us to invoke \eqref{cbds}. Then through the local minimality of $z_R$, along with \eqref{twoeta} and \eqref{vrgrad}, we find that
\begin{align}
 &E_R(z_R,B_{r^*})\leq E_R(v_R,B_{r^*})=\int_{\mathcal{A}_{r^*,r^*-\frac{1}{R}}} 
 RW(v_R)+\frac{1}{2R}\abs{\nabla v_R}^2\,dx\nonumber\\
 &\leq \int_{r^*-\frac{1}{R}}^{r^*}\int_{\partial B_r}\left\{R\lm(r)W\left(z_R\big(r^*\frac{x}{\abs{x}}\big)\right)+\right.\nonumber\\
 &\left.\frac{1}{2R}\left(2\abs{\nabla  z_R\left(r^*\frac{x}{\abs{x}}\right)}^2+\frac{2R^2}{b}
 W(z_R\left(r^*\frac{x}{\abs{x}}\right)\right)\right\}
 \,ds\,dr\leq 2\eta\,(\max\{2,1+\frac{1}{b}\})\frac{1}{R}.\nonumber\\
 \label{smaller}
\end{align}
Next we wish to argue that the maximum of the quantity $\abs{z_R(x)-p_\ell}$ over the set $\overline{B_{r^*}}$ must occur for $x\in \partial B_{r^*}$. We will argue by contradiction. There are two cases to consider:\\
\noindent Case 1: The maximum occurs at a point $x^*\in B_{r^*}$ such that
$\abs{z_R(x^*)-p_\ell}< \beta$. We see this is impossible through an appeal to the maximum principle applied to the function $f(x):=\frac{1}{2}\abs{z_R(x)-p_\ell}^2$. Indeed, a simple calculation yields that
\[
\Delta f(x)=R^2\,\nabla_u W(z_R)\cdot (z_R-p_\ell)+\abs{\nabla z_R}^2>0,
\]
in light of the strict convexity of $W(q)$ when $\abs{q-p_\ell}\leq \beta$. \\
\noindent Case 2: The maximum occurs at a point $x^*\in B_{r^*}$ such that
$\abs{z_R(x^*)-p_\ell}\geq \beta$. Since necessarily any local minimizer satisfies a gradient estimate $\abs{\nabla z_R}\leq C_1R$ for some constant $C_1>0$, it follows that
\[
\abs{z_R(x)-p_\ell}> \frac{\beta}{2}\quad \mbox{for all}\;x\;\mbox{such that}\;\abs{x-x^*}<\frac{\beta}{2C_1R}.
\]
Thus, denoting 
\[
C_\beta:=\min_{\{q\in\R^2:\,{\rm dist}\,(q,P)\geq\frac{\beta}{2}\}}W(q)>0,
\]
we obtain
\[
E_R(z_R,B_{r^*})\geq \int_{\{x:\,\abs{x-x^*}<\frac{\beta}{2C_1R}\}}RW(z_R)\,dx\geq
C_\beta R\pi\left(\frac{\beta}{2C_1R}\right)^2=\frac{\pi \beta^2 C_\beta}{4C_1^2 }\,\frac{1}{R},
\]
which will contradict \eqref{smaller}, if in addition to \eqref{firsteta}, we insist that $\eta$ satisfies, say
\beq
2\eta\max\{2,1+\frac{1}{b}\}\leq \frac{\pi \beta^2 C_\beta}{5C_1^2 }.\label{finalone}
\eeq
Hence, assuming $\eta$ satisfies \eqref{firsteta} and \eqref{finalone}, we have argued that the maximum of $\abs{z_R-p_\ell}$ on $B_{r^*}$ must occur on $\partial B_{r^*}$, and so from   \eqref{compare} and \eqref{threeeta}, the conclusion of the Proposition follows.
\end{proof}

An easy consequence of this result is 

\bprop\label{uniform} Assume $U:\R^2\to\R^2$ is a local minimizer of $E$ and $\{R_j\}\to\infty$ is such that  $\{U_{R_j}\}$ converges in $L^1_{\rm loc}$ to $u_0\in BV_{{\rm loc}}(\R^2;P)$. Then 
$\{U_{R_j}\}$ converges to $u_0$  locally uniformly outside the support of $\nabla u_0$.
\eprop
\begin{proof}
Assume $x_0$ does not belong to the support of $\nabla u_0$. We need to prove that $U_{R_j}$ converges uniformly to $u_0$ in a neighbourhood of $x_0$.

The limit $u_0$ is identically equal to one of the wells, say $p_1$, in a ball $B_r(x_0)$ for some $r>0$. Using Fatou's Lemma as in the proof of Lemma~\ref{centered} there exists a radius $t\in(0,r)$ and a subsequence still denoted $\{R_j\}$ such that 
$$\lim_{R_j\to \infty}\|U_{R_j} - p_1\|_{L^1(\partial B_t)} =0,\quad \limsup_{R_j\to \infty} E_{R_j}(U_{R_j},\partial B_t) < +\infty.$$
Since $U_{R_j}$ minimizes $E_{R_j}$ and satisfies \eqref{gamcon} and \eqref{DirData} on $\partial B_t$, we can apply condition \eqref{reco} of Theorem \ref{Gazoulis} to assert that
\[
E_{R_j}(U_{R_j},B_t)\to E^h_0(u_0,B_t)=0\;\mbox{as}\;R_j\to\infty,
\]
where $h$ is the trace of $u_0$ on $\partial B_t$, that is, $h=p_1$.
This allows us to apply Proposition \ref{newclearing} to conclude that $U_{R_j}$ is converging uniformly to $p_1$ on $B_{t/2}$. Since the subsequential limit is unique, the whole sequence converges uniformly to $p_1$ on $B_{t/2}$, proving the proposition.\end{proof}

Now we prove 
\bprop\label{third} Assume $U:\R^2\to\R^2$ is a local minimizer of $E$ in the sense of \eqref{italian} such that ${\rm {dist}}(U(0),\Lambda)>0$ for $\Lambda$ given by \eqref{Lambda}.
For any sequence $\{R_j\}\to\infty$, let $\{R_{j_k}\}$ and $u_0\in BV_{{\rm loc}}(\R^2;P)$ be  any subsequence and subsequential limit of $\{U_{R_{j_k}}\}$, guaranteed to exist by Proposition \ref{compactness}. Then $u_0$ takes the form \eqref{best}.
\eprop
\begin{proof} We need to rule out \eqref{con} and \eqref{half}.

Suppose first, by way of contradiction, that $u_0\equiv p_\ell$ for some $\ell\in\{1,2,3\}$. Then Proposition~\ref{uniform} implies that $U_{R_{j_k}}(0)\to U(0)=p_\ell$, contradicting ${\rm {dist}}(U(0),\Lambda)>0$. Thus, possibility \eqref{con} is eliminated.

Next, we suppose by way of contradiction that $U_{R_{j_k}}\to u_0$ in $L^1_{\rm{loc}}$ for $u_0$ satisfying \eqref{half} for some $i,\ell\in\{1,2,3\}$ with $i\not=\ell.$ With no loss of generality we will take $i=1,\,\ell=2$ and the halfplane $H$ to be $\{(x_1,x_2)\in\R^2:\,x_2<0\}$ so that our contradiction hypothesis takes the form
\begin{equation}
U_{R_{j_k}}\stackrel{L^1_{\rm{loc}}}{\longrightarrow} u_0
=\left\{\begin{matrix}p_1&\;\mbox{in}\;\{(x_1,x_2):\;x_2<0\}\\ p_2&\;\mbox{in}\;\{(x_1,x_2):\;x_2>0\}.\end{matrix}\right.\label{concon}
\end{equation}

This possibility is ruled out if we prove Theorem~\ref{halfplanes}, since the latter  implies that $U(x_1,x_2) = \zeta_{12}(x_2+\Delta)$ for some $\Delta\in \R$ and therefore $U(0)\in\zeta_{12}(\R)$, which contradicts the hypothesis that ${\rm {dist}}(U(0),\Lambda)>0$ for $\Lambda$ given by \eqref{Lambda}.
\end{proof}

\begin{proof}[Proof of Theorem~\ref{halfplanes}]
We assume that $U$ is a locally minimizing entire solution and that $U_{R_j}$ converges as $j\to+\infty$ for some subsequence for $R_j\to+\infty$ to the function $u_0$ defined in \eqref{thehalf}. We break the proof that $U(x_1,x_2) = \zeta_{12}(x_2+\Delta)$ for some $\Delta\in\R$ into several steps.
\vskip.1in
\noindent
{\bf 1.} We begin by identifying a circle on which $U$ has ``well-controlled" boundary values. To this end, we note that the argument in the proof of Proposition \ref{compactness} leading to \eqref{DirBd} and \eqref{limdir} applies just as well to assert the existence of {\it two} values, $1\leq \lm_1<\lm_2\leq 2$, for which these two properties hold on both $\partial B_{\lm_1}$ and $\partial B_{\lm_2}$. We then let $\mathcal{A}_{\lm_2,\lm_1}$ denote the annulus $B_{\lm_2}\setminus B_{\lm_1}$, and invoke the assumption \eqref{concon}. It follows from the $\Gamma$-convergence of $\tilde{E}_{R_j}(\cdot,\mathcal{A}_{\lm_2,\lm_1})$ to $E^h_0(\cdot,\mathcal{A}_{\lm_2,\lm_1})$ with $h=$ trace of $u_0$ on $\partial \mathcal{A}_{\lm_2,\lm_1}$, along with the minimality of $U_{R_j}$ in the annulus,  that
\beq
E_{R_j}(U_{R_j},\mathcal{A}_{\lm_2,\lm_1})\to E^h(u_0,\mathcal{A}_{\lm_2,\lm_1})=2(\lm_2-\lm_1)c_{12}\;\mbox{as}\;R_j\to\infty.
\label{annul}
\eeq
 Rewriting this in terms of $U$ we have that
\beq
E(U, \mathcal{A}_{\lm_2 R_j,\lm_1 R_j})=2(\lm_2-\lm_1)  c_{12}R_j+o(R_j)\;\mbox{as}\;R_j\to \infty,\label{bcon}
\eeq
or equivalently,
\[
%\frac{1}{(\lm_2-\lm_1) R_j}\int_{\lm_1 R_j}^{\lm_2 R_j}\int_{\partial B_r} W(U)+\frac{1}{2}\abs{\nabla U}^2\,ds\,dr= 2 c_{12}+o(1).
\frac{1}{(\lm_2-\lm_1) R_j}\int_{\lm_1 R_j}^{\lm_2 R_j}E\left(U,\partial B_r\right)\,dr= 2 c_{12}+o(1).
\]
Thus, by the Mean Value Theorem, there must exist a sequence of radii $\{\rho_j\}\to\infty$ with $\rho_j\in (\lm_1 R_j,\lm_2 R_j)$ such that
\beq
E\left(U,\partial B_{\rho_j}\right) = 2 c_{12}+o(1).\label{gbc}
\eeq

From Proposition~\ref{uniform}, we have that for any $\tau>0$
\beq
\max\left\{\abs{U(x)-p_1}:\,x\in  \overline{ B}_{\rho_j}\cap \{(x_1,x_2):\,x_2\geq \rho_j\tau\}\right\}\to 0
\label{unif1}
\eeq
and
\beq
\max\left\{\abs{U(x)-p_2}:\,x\in  \overline{ B}_{\rho_j}\cap \{(x_1,x_2):\,x_2\leq -\rho_j\tau\}\right\}\to 0\quad\mbox{as}\;\rho_j\to\infty.
\label{unif2}
\eeq
\vskip.1in
\noindent
{\bf 2.} We will next argue that the restriction of $U$ to the circle $\partial B_{\rho_j}$ is approaching two copies of the geodesic $\zeta_{12}$ as $\rho_j\to\infty$.  
We denote by $\partial B_{\rho_j}^+$ the right half-circle of $\partial B_{\rho_j}$, and by $\partial B_{\rho_j}^-$ the left half-circle.
From \eqref{unif1} and \eqref{unif2} and the continuity of the metric distance $(p,q)\mapsto d(p,q)$ (cf. \eqref{Wdist}), it follows that 
\beq
%\int_{\partial B_{\rho_j}^+} W(U)+\frac{1}{2}\abs{\nabla U}^2\,ds
E\left(U,\partial B_{\rho_j}^+\right)
\geq \sqrt{2}\int_{\partial B_{\rho_j}^+} \sqrt{W(U)}\abs{\frac{\partial U}{\partial s}}\,ds\geq  d(p_1,p_2)-o(1)=c_{12}-o(1)\;\mbox{as}\;\rho_j\to\infty\label{j1}
\eeq
and similarly, 
\beq
%\int_{\partial B_{\rho_j}^-} W(U)+\frac{1}{2}\abs{\nabla U}^2\,ds
E\left(U,\partial B_{\rho_j}^-\right)
\geq \sqrt{2}\int_{\partial B_{\rho_j}^-} \sqrt{W(U)}\abs{\frac{\partial U}{\partial s}}\,ds\geq  d(p_1,p_2)-o(1)=c_{12}-o(1)\;\mbox{as}\;\rho_j\to\infty\label{j2}
\eeq
Combining these last two inequalities with \eqref{gbc}, we observe that, in fact, we have
\beq
%\int_{\partial B_{\rho_j}^+} W(U)+\frac{1}{2}\abs{\nabla U}^2\,ds
E\left(U,\partial B_{\rho_j}^+\right)
\to c_{12}, \quad
\mbox{as} \;\rho_j\to\infty\label{minseq1}
\eeq
and 
\beq
%\int_{\partial B_{\rho_j}^-} W(U)+\frac{1}{2}\abs{\nabla U}^2\,ds
E\left(U,\partial B_{\rho_j}^-\right)
\to c_{12} \quad
\mbox{as} \;\rho_j\to\infty.\label{minseq2}
\eeq
Now, following the general scheme in \cite{AFS}, sect. 2.3 we fix a positive number $d_0$ less than say, half the minimal distance between any two of the three wells, and define the energy level $W_0$ via
\beq
W_0:=\min\bigg\{W(p):\,\abs{p-p_\ell}=d_0,\; \ell=1,2,3\bigg\}>0.\label{Wzero}
\eeq 
 It follows from \eqref{unif1} and \eqref{unif2} that there must exist a point $x_j^+\in\partial B_{\rho_j}$ such that $W\left(U(x_j^+)\right)=W_0.$
 Then for $\theta$  denoting the polar angle made with the positive $x_1$-axis, we introduce the angle $\theta_j^+$ via $x_j^+=\rho_je^{i\theta_j^+}$.
We point out that necessarily 
\beq
\theta_j^+\to 0\;\mbox{as}\;j\to\infty,\label{antipod}
\eeq
since $W\left(U(\rho_je^{i\theta})\right)\to 0$ at angles $\theta$ bounded away from zero in light of \eqref{unif1} and \eqref{unif2}.

Then we introduce an arclength coordinate $s$ along $\partial B_{\rho_j}^+$  with $s=0$ corresponding to this $x_j$ via 
\beq
s:=\rho_j(\theta-\theta_j^+)\label{stheta},
\eeq and define the continuous extension, say $\tilde{U}_j:(-\infty,\infty)\to\R^2$  of $U$ along $\partial B_{\rho_j}^+$, expressed as a function of arclength variable \eqref{stheta}, through the formula
\[
\tilde{U}_j(s)=\left\{\begin{matrix} p_1&\mbox{for}\;s\geq \rho_j(\pi/2-\theta_j^+)+1\\
\mbox{linear}&\;\mbox{for}\;\rho_j(\pi/2-\theta_j^+)<s<\rho_j(\pi/2-\theta_j^+)+1\\
U\left( \rho_je^{i(\frac{s}{\rho_j}+\theta_j^+)}\right)&\;\mbox{for}\;\rho_j(-\pi/2-\theta_j^+)\leq s\leq \rho_j(\pi/2-\theta_j^+)\\
\mbox{linear}&\;\mbox{for}\;\rho_j(-\pi/2-\theta_j^+)-1<s<\rho_j(-\pi/2-\theta_j^+)\\
p_2&\mbox{for}\;s\leq \rho_j(-\pi/2-\theta_j^+)-1,
\end{matrix}\right.
\]
so that, in particular, we have 
\beq
W(\tilde{U}_j(0))=W_0.\label{centered1}
\eeq 
Now again appealing to  \eqref{unif1} and \eqref{unif2}, we observe that the energy over the intervals of linear interpolation vanish in the limit; that is,
\[
\lim_{j\to\infty}
%\int_{\{s:\,\rho_j(\pi/2-\theta_j^+)<s<\rho_j(\pi/2-\theta_j^+)+1\}}W(\tilde{U}_j)+\frac{1}{2}\abs{\frac{d\tilde{U}_j}{ds}}^2\,ds
E\left(\tilde{U}_j,[\rho_j(\pi/2-\theta_j^+),\rho_j(\pi/2-\theta_j^+)+1]\right)
=0,
\]
with a similar result for the integral over the interval $[\rho_j(-\pi/2-\theta_j^+)-1,\rho_j(-\pi/2-\theta_j^+)]$.
Hence, through \eqref{minseq1}, we see that $\{\tilde{U}_j\}$ constitutes a minimizing sequence for \eqref{other}. Invoking \eqref{centered1},
it is then straightforward to establish that 
\beq
\tilde{U}_j\to \zeta_{12}\quad\mbox{in}\;H^1_{\rm{loc}}(\R)\;\mbox{as}\;j\to\infty,\label{Honegeo}
\eeq
with $W(\zeta_{12}(0))=W_0$ setting the particular translate of the heteroclinic $\zeta_{12}$. Again we refer to \cite{AFS} for details. (We recall here that we are assuming uniqueness of the three heteroclinic connections.) We then note that using \eqref{minseq2}, we can apply precisely the same argument along the left half-circle $\partial B_{\rho_j}^-$ to get $H^1_{\rm{loc}}$-convergence to $\zeta_{12}$ analogous to \eqref{Honegeo} there as well. In particular, in analogy with  \eqref{antipod} and \eqref{centered1}, we note that there exists an angle made with the negative $x_1$-axis, which we denote by $\theta_j^-$, that plays the same role as did $\theta_j^+$; namely, 
\beq
W\left(U\left(\rho_je^{i(\pi-\theta_j^-)}\right)\right)=W_0\quad\mbox{and}\quad 
\theta_j^-\to 0\;\mbox{as}\;j\to\infty.\label{antipod2}
\eeq

Also, referring back to  \eqref{unif1} and \eqref{unif2}, with, say $\tau=1/2$,  it follows that 
\beq
\abs{U\left(\rho_je^{i\theta}\right)-p_1}\to 0\;\mbox{for}\;\frac{\pi}{6}\leq \theta\leq \frac{5}{6}\pi\quad\mbox{and}\quad
\abs{U\left(\rho_je^{i\theta}\right)-p_2}\to 0\;\mbox{for}\;-\frac{5}{6}\pi\leq \theta\leq -\frac{\pi}{6}.\label{ponetwo}
\eeq
\vskip.1in\noindent
{\bf 3.} Letting $L_j$ denote the line passing through the two points $x_j^+=\rho_je^{i \theta_j^+}$ and $x_j^-:=\rho_je^{i(\pi-\theta_j^-)}$, we define the sequence $V_j:\R^2\to\R^2$ via
\beq
V_j(x):=\zeta_{12}\left({\rm{dist}}\,(x,L_j)\right).\label{zetaslanted}
\eeq
Our goal in this step is to interpolate between $U$ on $\partial B_{\rho_j}$ and $V_j$ on $\partial B_{\rho_j-1}$ so that the energetic cost in the annulus between these two circles is no greater than $2c_{12}+o(1)$. Again we will focus on the right half-annulus, with a similar calculation applying to the left half-annulus.

To this end, we first recall the exponential approach of $\zeta_{12}=\zeta_{12}(t)$ to $p_1$ for $t\gg1$ and to $p_2$ for $t\ll -1$, cf. \eqref{expdecay}. Fixing any $\eta>0$, it follows that we can find an interval, say  $[-a_\eta,a_\eta]$,  such that
\beq
%c_{12}-\eta\leq \int_{-a_\eta  }^{a_\eta} W(\zeta_{12})+\frac{1}{2}\abs{\zeta'_{12}}^2\,dt
E\left(\zeta_{12},[-a_\eta,a_\eta]\right)
\leq c_{12},
\label{w1}
\eeq
and such that
\beq\abs{\zeta_{12}(t)-p_1}<\eta\;\mbox{for}\;t>a_\eta,\quad\mbox{and}\quad \abs{\zeta_{12}(t)-p_2}<\eta\;\mbox{for}\;t<-a_\eta.
\label{w11}
\eeq

Then in view of the $H^1$-convergence of $\tilde{U}_j$ to $\zeta_{12}$ for $s\in [-a_\eta,a_\eta]$ guaranteed by \eqref{Honegeo}, we can assert that for $\rho_j$ large enough, one also has
\beq
    c_{12}-\eta\leq 
    %\int_{-a_\eta   }^{a_\eta}  W(\tilde{U}_j)+\frac{1}{2}\abs{\tilde{U}_j\,'}^2\,ds
    E\left(\tilde{U}_j,[-a_\eta,a_\eta]\right)
    \leq c_{12}+\eta.
\label{w2}
\eeq
As consequence of \eqref{w1} and \eqref{w2}, along with \eqref{other}, \eqref{minseq1} and \eqref{minseq2}, it also follows that 
\beq
%\int_{\left\{s:\,\abs{s}>a_{\eta}\right\}    } W(\zeta_{12})+\frac{1}{2}\abs{\zeta'_{12}}^2\,ds
E\left(\zeta_{12},\R\setminus [-a_\eta,a_\eta]\right)
\leq \eta
\label{w3}
\eeq
and 
\beq
%\int_{\left\{s:\,\abs{s}>a_{\eta}\right\}    }  W(\tilde{U}_j)+\frac{1}{2}\abs{\tilde{U}_j\,'}^2\,ds
E\left(\tilde{U}_j,\R\setminus[-a_\eta,a_\eta]\right)
\leq \eta.
\label{w4}
\eeq
Of course, the analog of estimates \eqref{w2} and \eqref{w4} hold with the boundary values of $U$ along $\partial B_{\rho_j}^+$ replaced by those along $\partial B_{\rho_j}^-$ as well.

Now we define the linear interpolation in the annulus $\{x:\,\rho_j-1\leq \abs{x}\leq \rho_j\}$  with $\lm_j(r):=r-\rho_j+1$, via the formula
\beq
Z_j(x):=\lm_j(\abs{x})U\left(\rho_j\frac{x}{\abs{x}}\right)+\left(1-\lm_j(\abs{x})\right)V_j(x).
\label{bigZdefn}
\eeq

We will divide up the energy of $Z_j$ in the right half-annulus into two parts as follows:
\begin{eqnarray}
&&\int\int_{{\rm{right\; half-annulus}}}W(Z_j)+\frac{1}{2}\abs{\nabla Z_j}^2\,dx=\nonumber\\
&&\int_{\rho_j-1}^{\rho_j}\int_{\left\{x\in \partial B_r^+:\,   {\rm{dist}}\,(x,L_j) <a_\eta\right\}    }
\left\{
\quad\cdot\quad \right\}ds\,dr+ \int_{\rho_j-1}^{\rho_j}\int_{\left\{x\in \partial B_r^+:\,   {\rm{dist}}\,(x,L_j) \geq a_\eta\right\}    }
\left\{
\quad\cdot\quad \right\}ds\,dr\nonumber\\
&&=:I+II.\label{onetwo}
\end{eqnarray}
Regarding integral $I$, we note that for any $r$ such that  $\rho_j-1<r<\rho_j$ and any $x\in\partial B_r$ such that ${\rm{dist}}\,(x,L_j) <a_\eta$,
if we denote by $s(x)$ the arclength along $\partial B_r$ from $x$ to $\partial B_r\cap L_j$, then one easily checks that
\beq
0\leq s(x)-{\rm{dist}}\,(x,L_j)=O\left(\frac{1}{\rho_j}\right).\label{comp}
\eeq
Combining this with the uniform convergence of $\tilde{U}_j$ to $\zeta_{12}$ on $[-a_\eta,a_\eta]$ guaranteed by \eqref{Honegeo}, we obtain that 
\beq
\abs{U\left(\rho_j\frac{x}{\abs{x}}\right)-V_j(x)}\leq \abs{U\left(\rho_je^{i(\frac{s(x)}{\rho_j}+\theta_j^+)}\right)-\zeta_{12}\left(s(x)\right)}
+\abs{\zeta_{12}\left(s(x)\right)-\zeta_{12}\left({\rm{dist}}\,(x,L_j)\right)}=o(1),\label{UVj}
\eeq
which then also implies that
\beq
Z_j(x)=\zeta_{12}\left(s(x)\right)+o(1).\label{Zz}
\eeq
Now
\[
\abs{\nabla U\left(\rho_j\frac{x}{\abs{x}}\right)}^2=\abs{\tilde{U}_j\,'\left(s(x)\right)}^2,
\]
and for $x$ in the domain of integration of integral $I$, with an appeal to \eqref{comp}, we have
\beq
\abs{\nabla V_j(x)}^2=\abs{\zeta_{12}'\left({\rm{dist}}\,(x,L_j)\right)}^2=\abs{\zeta_{12}'(s)}^2+o(1).\label{Vz}
\eeq
Therefore, since
\beq
\nabla Z_j(x)=\lm_j(\abs{x})\nabla  U\left(\rho_j\frac{x}{\abs{x}}\right)+\left(1-\lm_j(\abs{x})\right)\nabla V_j(x)+\frac{x}{\abs{x}}\otimes 
\left(U\left(\rho_j\frac{x}{\abs{x}}\right)-V_j(x)\right),\label{Zgrad}
\eeq
we may compute that
\begin{eqnarray}
&&\abs{\nabla Z_j(x)}^2\leq
\abs{\lm_j(\abs{x})\nabla  U\left(\rho_j\frac{x}{\abs{x}}\right)+\left(1-\lm_j(\abs{x})\right)\nabla V_j(x)}^2\nonumber\\
&& +C\left\{\abs{U\left(\rho_j\frac{x}{\abs{x}}\right)-V_j(x)}^2  + \abs{U\left(\rho_j\frac{x}{\abs{x}}\right)-V_j(x)}\left(\abs{ \nabla  U\left(\rho_j\frac{x}{\abs{x}}\right)}+\abs{\nabla V_j(x)}\right)\right\}\nonumber\\
&&
\leq 
\lm_j(\abs{x})\abs{\nabla  U\left(\rho_j\frac{x}{\abs{x}}\right)}^2+\left(1-\lm_j(\abs{x})\right)\abs{\nabla V_j(x)}^2
+o(1),\label{string}
\end{eqnarray}
where the last inequality follows from the convexity of $\abs{\,\cdot\,}^2$, along with the use of \eqref{comp} and \eqref{UVj}, after noting that both $\abs{\nabla U}$ and $\abs{\nabla V_j}$ are uniformly bounded.

As a consequence of \eqref{Zz}, \eqref{string} and another appeal to
 \eqref{Honegeo} we find that 
 \begin{eqnarray*}
&& \int_{\rho_j-1}^{\rho_j}\int_{\left\{x\in \partial B_r^+:\,   {\rm{dist}}\,(x,L_j) <a_\eta\right\}    }
W(Z_j(x))+\frac{1}{2}\abs{\nabla Z_j(x)}^2\,ds\,dr\leq\\
&&  \int_{\rho_j-1}^{\rho_j}\int_{-a_\eta}^{a_\eta} W(\zeta_{12}(s))+\frac{1}{2}\left\{
\lm_j(r)\abs{\tilde{U}_j\,'(s)}^2+\left(1-\lm_j(r)\right)\abs{ \zeta_{12}'(s)}^2\right\}\,ds\,dr+o(1)\\
&&= \int_{\rho_j-1}^{\rho_j}\int_{-a_\eta}^{a_\eta}W(\zeta_{12}(s))+\frac{1}{2}
\abs{ \zeta_{12}'(s)}^2\,ds\,dr+o(1).
\end{eqnarray*}
Hence, it follows from \eqref{w1} that
\beq
I\leq c_{12}+\eta,\label{ione}
\eeq
with a corresponding inequality holding for the energy of $\{Z_j\}$ over the region in the left portion of the annulus given by
\[
\left\{x=(x_1,x_2):\;\rho_j-1<\abs{x}<\rho_j,\;{\rm{dist}}\,(x,L_j)<\eta,\;x_1<0\right\}.
\]

It remains to estimate integral $II$ in \eqref{onetwo}. We will argue that this integral is $o(1)$ by relying on \eqref{w3} and \eqref{w4}. We first claim that 
\beq
\int_{\rho_j-1}^{\rho_j}\int_{\left\{x\in \partial B_r^+:\,   {\rm{dist}}\,(x,L_j) \geq a_\eta\right\}    }
W(Z_j(x) )ds\,dr=O(\eta).\label{clam}
\eeq

With an eye towards appealing to \eqref{convexngd}, we observe that necessarily for any $x\in\partial B_{\rho_j}$ such that $ {\rm{dist}}\,(x,L_j) \geq a_\eta$, one has either
$\abs{U(x)-p_1}< \beta$ or  $\abs{U(x)-p_2}< \beta$. Otherwise, the energetic cost incurred by transitioning along $\partial B_{\rho_j}$ from $\abs{U(x)-p_\ell}= \beta$ to 
$ U(x)\approx p_\ell$ would be $O(1),$ violating \eqref{w4}, a transition that must occur in light of  \eqref{ponetwo}.

Since 
 \eqref{w11} implies that $V_j(x)$ also takes values in a region of convexity of $W$ for $x\in\partial B_{\rho_j}$ such that $ {\rm{dist}}\,(x,L_j) \geq a_\eta$, we have that
\beq
W(Z_j(x))\leq \lm_j(\abs{x})W\left( U\left(\rho_j\frac{x}{\abs{x}}\right)\right)+\left(1-\lm_j(\abs{x})\right)W(V_j(x))\quad\mbox{for such an}\;x.
\eeq
 Hence, with an appeal to
 \eqref{w3} and \eqref{w4}, we obtain claim \eqref{clam}.

To show that integral $II$ in \eqref{onetwo} is small, we still must estimate the integral of $\abs{\nabla Z_j}^2.$ Here we note from \eqref{Zgrad} that
\begin{eqnarray}
&&\abs{\nabla Z_j}^2\leq C\left\{\left[\lm_j(\abs{x})\nabla  U\left(\rho_j\frac{x}{\abs{x}}\right)+\left(1-\lm_j(\abs{x})\right)\nabla V_j(x) \right]^2 + \abs{    U\left(\rho_j\frac{x}{\abs{x}}\right)-V_j(x)  }^2\right\}\nonumber\\
&&
\leq C\left\{  \lm_j(\abs{x})\abs{\nabla  U\left(\rho_j\frac{x}{\abs{x}}\right) }^2+\left(1-\lm_j(\abs{x}\right)\abs{\nabla V_j(x)}^2+    \abs{    U\left(\rho_j\frac{x}{\abs{x}}\right)-V_j(x)  }^2                  \right\}.\label{Ztri}
\end{eqnarray}
Integrating this expression over the set 
\[
\{x=(x_1,x_2):\,\rho_j-1<\abs{x}<\rho_j,\; {\rm{dist}}\,(x,L_j) \geq a_\eta,x_1>0\},
\]
we can use \eqref{w3} and \eqref{w4} once again to show that the first two terms in \eqref{Ztri} integrate to $O(\eta)$. 

The third term can be handled in the same manner as was done for the integral of $W(Z_j)$. That is, we split the integral into the set where $U\left(\rho_j\frac{x}{\abs{x}}\right)$ is far from both $p_1$ and $p_2$ and where it is close to one of these wells. We know from \eqref{w11} that $V_j(x)$ is near $p_1$ or $p_2$ for this domain of integration and therefore the measure of the set where $U\left(\rho_j\frac{x}{\abs{x}}\right)$ is far from the wells must be small in order not to violate \eqref{w4}. Then, on any set where it is near  to $p_1$ or $p_2$, we have
\[
\abs{    U\left(\rho_j\frac{x}{\abs{x}}\right)-V_j(x)  }^2\leq 2\left(\abs{    U\left(\rho_j\frac{x}{\abs{x}}\right)-p_\ell }^2+
\abs{    V_j(x)-p_\ell  }^2\right)\quad\mbox{for either}\;\ell=1\;\mbox{or}\;2,
\]
and so the quantity $\abs{    U\left(\rho_j\frac{x}{\abs{x}}\right)-V_j(x)  }^2 $ is controlled by the sum of the integrals of $W\left(U\left(\rho_j\frac{x}{\abs{x}}\right)\right)$ and $W(V_j(x))$. Hence, by \eqref{w3} and \eqref{w4} it must also integrate to $O(\eta)$.

Since the analysis of $I$ leading to \eqref{ione} holds for any $\eta>0$ provided $\rho_j$ is sufficiently large, as does this just completed analysis of integral $II$, we finally conclude that the interpolating sequence $\{Z_j\}$ satisfies the bound
\[
\int_{\{x:\,\rho_j-1<\abs{x}<\rho_j,\;x_1>0\}}W(Z_j)+\frac{1}{2}\abs{\nabla Z_j}^2\,dx\leq c_{12}+o(1).
\]
The argument leading to the same estimate for the energy of $Z_j$ taken over the set
$\{x:\,\rho_j-1<\abs{x}<\rho_j,\;x_1<0\}$ is identical, and so we arrive at the estimate
\beq
%\int_{\{x:\,\rho_j-1<\abs{x}<\rho_j\}}W(Z_j)+\frac{1}{2}\abs{\nabla Z_j}^2\,dx
E\left(Z_j,B_{\rho_j}\setminus B_{\rho_j - 1}\right)
\leq 2c_{12}+o(1).\label{annulusest}
\eeq
\vskip.1in\noindent
{\bf 4.} Having interpolated between the boundary values of $U$ on $\partial B_{\rho_j}$ and those of $V_j$ on $\partial B_{\rho_j-1}$ with a cost bounded as in \eqref{annulusest}, we can now appeal to the local minimality of $U$ to assert that
\[
%\int_{B_{\rho_j}}W(U)+\frac{1}{2}\abs{\nabla U}^2\,dx
E\left(U,B_{\rho_j}\right)\leq 2c_{12}+
%\int_{B_{\rho_j-1}}W(V_j)+\frac{1}{2}\abs{\nabla V_j}^2\,dx
E\left(V_j,B_{\rho_j - 1}\right)
+o(1),
\]
where we recall that $V_j$ is defined through \eqref{zetaslanted}.
If we now consider a coordinate system $(z_1,z_2)$ with $z_1$-axis coinciding with the line $L_j$,  $z_2$-axis orthogonal to it, and with origin at the midpoint of the line segment $L_j\cap B_{\rho_j}$, then it follows immediately from the definition of $\zeta_{12}$ and \eqref{zetaslanted} that
\beq
E\left(U,B_{\rho_j}\right) \leq 2c_{12}+\int_{-\frac{1}{2}\mH^1(L_j\cap B_{\rho_j-1})}^{\frac{1}{2}\mH^1(L_j\cap B_{\rho_j-1})} E\left(\zeta_{12},\R\right)\,dz_1 + o(1) = c_{12}\,\mH^1\left(L_j\cap B_{\rho_j}\right)+o(1).
\label{upbd}
\eeq
% \begin{eqnarray}
% &&
% \int_{B_{\rho_j}}W(U)+\frac{1}{2}\abs{\nabla U}^2\,dx\leq 2c_{12}+\int_{-\frac{1}{2}\mH^1(L_j\cap B_{\rho_j-1})}^{\frac{1}{2}\mH^1(L_j\cap B_{\rho_j-1})}
% \int_{-\infty}^{\infty}W\left(\zeta_{12}(z_2)\right)+\frac{1}{2}\abs{ \zeta'(z_2)}^2\,dz_2\,dz_1+o(1)\nonumber\\
% &&=c_{12}\,\mH^1\left(L_j\cap B_{\rho_j}\right)+o(1).\label{upbd}
% \end{eqnarray}
Here we are using the fact that \[
\mH^1\left(L_j\cap\{x:\;\rho_j-1<\abs{x}<\rho_j\}\right)\leq 2+o(1),
\]
since we recall that the line $L_j$ meets $\partial B_{\rho_j}$ at the points $x_J^+=\rho_je^{i \theta_j^+}$ and $x_j^-=\rho_je^{i(\pi-\theta_j^-)}$ with $\theta_j^+$ and $\theta_j^-$ both approaching zero by \eqref{antipod} and \eqref{antipod2}. See Figure \ref{heteroball}.
\begin{figure}[H]
	\centering
	\includegraphics[width = 0.6\textwidth]{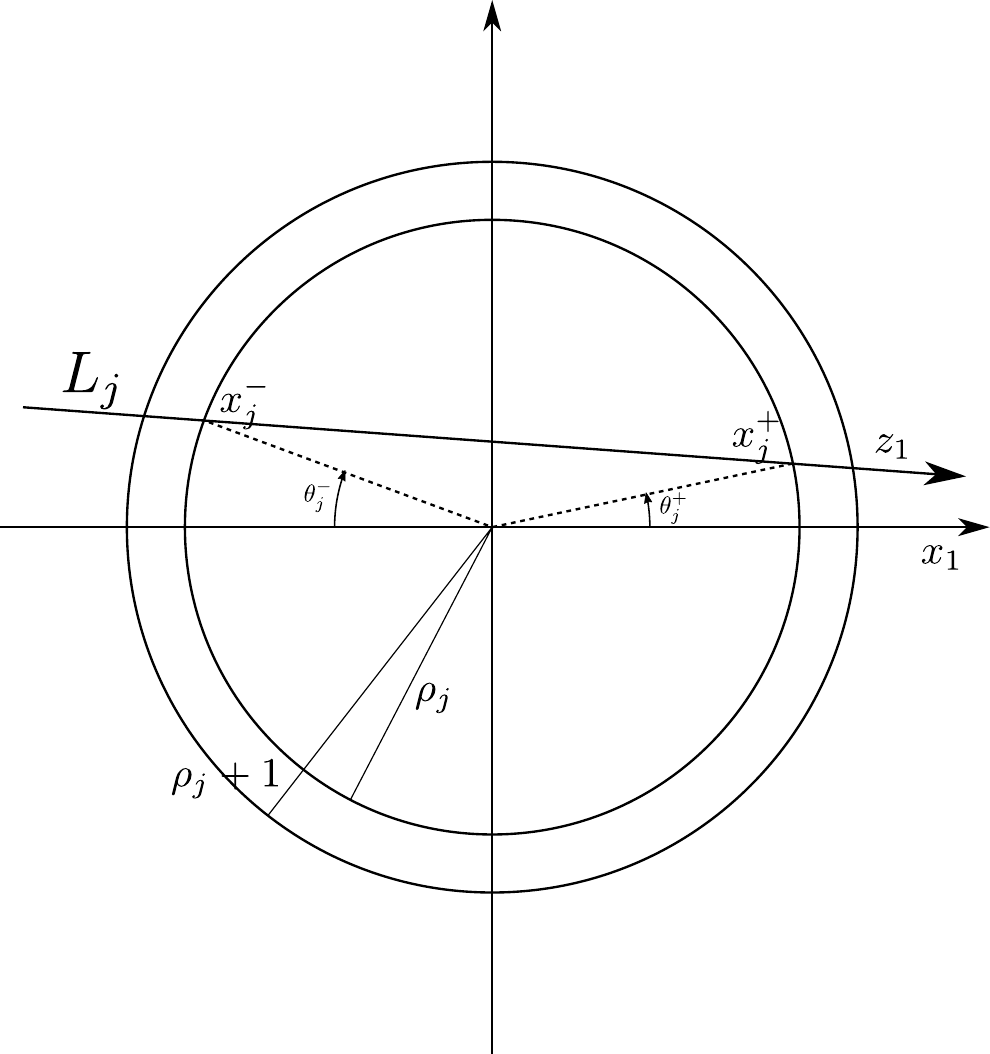}
	\caption{The coordinate system based on the line $L_j$ passing through the points $x_j^+=\rho_je^{i\theta_j^+}$ and $x_j^-=\rho_je^{i(\pi-\theta_j^-)}$ defined in \eqref{antipod} and \eqref{antipod2}.}
	\label{heteroball}
\end{figure}
 \vskip.1in\noindent
 {\bf 5.} We conclude the proof of Proposition \ref{third} with a lower bound for the energy of $U$ on $B_{\rho_j}$ that will contradict the upper bound \eqref{upbd}, thus eliminating the possibility that \eqref{concon} can occur. To this end, we introduce an interpolation, say $\tilde{Z}_j$, on the annulus $\{x:\,\rho_j\leq \abs{x}\leq \rho_j+1\}$ between the values of $U$ on $\partial B_{\rho_j}$ and those of $V_j$ given by \eqref{zetaslanted} on $\partial B_{\rho_j+1}.$
The formula for the sequence $\{\tilde{Z}_j\}$ is identical to that of $\{Z_j\}$ given in \eqref{bigZdefn}, with the exception that $\lm_j(r)$ is now replaced by  $\tilde{\lm}_j(r):=\rho_j+1-r.$ Then the entire argument that led to the energy bound \eqref{annulusest} applies equally well to $\{\tilde{Z}_j\}$, to establish that
\beq
% \int_{\{x:\,\rho_j<\abs{x}<\rho_j+1\}}W(\tilde{Z}_j)+\frac{1}{2}\abs{\nabla \tilde{Z}_j}^2\,dx
E\left(\tilde{Z}_j,B_{\rho_j+1}\setminus B_{\rho_j}\right)
\leq 2c_{12}+o(1).\label{annulusest2}
\eeq

As a consequence of \eqref{annulusest2}, if we now define the sequence $\{\mathcal{U}_j\}$ on $B_{\rho_j+1}$ via 
\[
\mathcal{U}_j(x):=\left\{\begin{matrix} U(x)&\;\mbox{for}\;x\in B_{\rho_j},\\
\tilde{Z}_j(x)&\;\mbox{for}\;x\in B_{\rho_j+1}\setminus B_{\rho_j},
\end{matrix}\right.
\]
then we have the lower energy bound
\beq
E\left(U,B_{\rho_j}\right)\geq E\left(\mathcal{U}_j,B_{\rho_j+1}\right)
%\int_{B_{\rho_j}}W(U)+\frac{1}{2}\abs{\nabla U}^2\,dx\geq
%\int_{B_{\rho_j+1}}W(\mathcal{U}_j)+\frac{1}{2}\abs{\nabla \mathcal{U}_j}^2\,dx
-2c_{12}-o(1).\label{lba}
\eeq
Next we will compute a lower bound for the integral on the right using a coordinate system $(z_1,z_2)$ very similar to the one introduced above \eqref{upbd}, where again $z_1$ is the arc length on the line $L_j$ but now with $z_1 = 0$, $z_2 = 0$ corresponding to the midpoint of $L_j\cap B_{\rho_j+1}$. For any $z_1\in \R$ the set $\{z_2\mid (z_1,z_2)\in B_{\rho_j+1}\}$ is an interval that we denote $\left(a_j(z_1),b_j(z_1)\right)$.

The line $L_j$ is at a distance $\delta_j = \rho_j\abs{\sin\bigg(\frac{\theta_j^- + \theta_j^+}{2}\bigg)}$ from the origin. 
Note that $\delta_j\ll \rho_j$ since $\abs{\theta_j^+}+\abs{\theta_j^-}\to 0.$. It follows that $B_{\rho_j+1}$ contains the set of points with coordinates $(z_1,z_2)$ such that 
\beq\label{zstrip} - \ell_j < z_1 < \ell_j, \quad -h_j(z_1) < z_2 < h_j(z_1),\eeq
where 
\begin{align}\label{zbounds}&\ell_j = \frac12 \mH^1\left(L_j\cap B_{\rho_j+1}\right) = \left((\rho_j+1)^2 - {\delta_j}^2\right)^{1/2}\nonumber\\
\mbox{and}\quad &h_j(z_1) =  \left((\rho_j+1)^2 - {z_1}^2\right)^{1/2} - \delta_j.\end{align}
In particular, 
\beq \label{segbound} h_j(z_1)\le \min\left(\abs{a_j(z_1)},b_j(z_1)\right).\eeq

Then we can write
\beq\label{slicing}E\left(\mathcal{U}_j, B_{\rho_j+1}\right) \geq \int_{B_{\rho_j+1}}\left|\frac{\partial\mathcal U_j}{\partial z}\right|^2+ \int_{-\ell_j}^{\ell_j} I(z_1) \,dz_1,\eeq
where $$I(z_1) = E\left(\mathcal{U}_j(z_1,\cdot),[a_j(z_1),b_j(z_1)]\right).$$
From the construction of $\mathcal{U}_j$ using $\tilde{Z_j}$, we have  that $\mathcal{U}_j(z_1,z_2) = \zeta_{12}(z_2)$ if $(z_1,z_2)\in\partial B_{\rho_j+1}$, that is if $z_2 = a_j(z_1)$ or $z_2 = b_j(z_1)$. Thus we may extend $\mathcal{U}_j(z_1,\cdot)$ continuously by setting $\mathcal{U}_j(z_1,z_2) = \zeta_{12}(z_2)$ if $z_2\notin \left(a_j(z_1),b_j(z_1)\right)$. 

Using the minimizing property of $\zeta_{12}$ with respect to its boundary conditions on any interval $[a,b]$ and its exponential decay as in \eqref{expdecay}, we have for any $z_1\in (-\ell_j,\ell_j)$ that 
\begin{eqnarray} 
I(z_1) &=& E\left(\mathcal{U}_j(z_1,\cdot),\R\right) - E\left(\mathcal{U}_j(z_1,\cdot),\R\setminus[a_j(z_1),b_j(z_1)]
\right)\nonumber\\
&=& E\left(\mathcal{U}_j(z_1,\cdot),\R\right) - E\left(\zeta_{12},\R\setminus[a_j(z_1),b_j(z_1)]\right)\nonumber\\
&\ge& E\left(\zeta_{12},\R\right) - E\left(\zeta_{12},\R\setminus[-h_j(z_1),h_j(z_1)]\right)\ge c_{12} - Ce^{-ch_j(z_1)},\nonumber\\
&\label{lowslice1}
\end{eqnarray}
where $c,C>0$ depend only on $W$.

We integrate \eqref{lowslice1} over $z_1\in (-\ell_j,\ell_j)$. In view of \eqref{slicing} we find that
$$E(\mathcal{U}_j, B_{\rho_j+1}) \ge \int_{B_{\rho_j+1}}\left|\frac{\partial\mathcal U_j}{\partial z_1}\right|^2+ 2 c_{12} \ell_j - C\int_{-\ell_j}^{\ell_j} e^{- c h_j(z_1)}\,dz_1.$$
It is straightforward to show, using \eqref{zbounds} and the fact that $\delta_j\ll \rho_j$ as $\rho_j\to +\infty$, that the last integral above is $o(1)$ and then, since 
$$2\ell_j = \mH^1\left(L_j\cap B_{\rho_j+1}\right) = \mH^1\left(L_j\cap B_{\rho_j}\right)+2+o(1), $$
we conclude that 
$$E(\mathcal{U}_j, B_{\rho_j+1}) \ge \int_{B_{\rho_j+1}}\left|\frac{\partial\mathcal U_j}{\partial z_1}\right|^2+c_{12}\,\mH^1\left(L_j\cap B_{\rho_j}\right)+2c_{12}+ - o(1),$$
In view of \eqref{lba} and the upper bound \eqref{upbd}, we deduce that 
\beq\label{noz}\int_{B_{\rho_j+1}}\left|\frac{\partial\mathcal U_j}{\partial z_1}\right|^2 = o(1).\eeq
 \vskip.1in\noindent
 {\bf 6.} We may now conclude. Going back to the original coordinates $(x_1,x_2)$ we have that $U = \mathcal U_j\circ \varphi_j$ on $B_{\rho_j}$, where $\varphi_j$ denotes the $(x_1,x_2)\to (z_1,z_2)$ map. But, as $j\to +\infty$, the rotational component of $\varphi_j$ converges to the identity since $\theta_j^+$ and $\theta_j^-$ both tend to $0$. Passing to the limit in \eqref{noz} we thus deduce that $U(x_1,x_2)$ does not depend on $x_1$. 
 
 Then, as a function of $x_2$ only, $U$ is a minimizing solution on $\R$, which converges to $p_1$ as $z_2\to -\infty$ and to $p_2$ as $x_2\to +\infty$. Thus there exists $\Delta\in\R$ such that $U(x_1,x_2) = \zeta_{12}(x_2+\Delta).$ The proof of Theorem~\ref{halfplanes} is complete.

 \end{proof}

 \section{Two partitioning problems: Regularity and comparison}\label{lastpf}

Our proof of the key estimate \eqref{keyest}
 in Proposition \ref{equipartition} relies upon the regularity theory for two partitioning problems, one fairly standard and the other perhaps not so standard. In this section we state the regularity theory for these problems, as established in \cite{MNov}. Then we state and prove a result relating the infima of these two problems.

For the convenience of the reader, we restate these two partitioning problems here:
\vskip.1in
\noindent
{\bf Problem 1.} Fix a function $h\in BV(\partial B;P)$. 
For any disjoint sets $S_1,S_2$ and $S_3$  of finite perimeter in $B$ such that
$\abs{B\setminus \cup S_\ell}=0$ define
\[
m_0:=\inf \left\{E_0(S_1,S_2,S_3):\;\cup S_\ell=B,\;\partial S_\ell\cap \partial B=h^{-1}(p_\ell)\;\mbox{for}\;\ell=1,2,3\right\},
\]
where $E_0$ is given by
\[
E_0(S_1,S_2,S_3)=t_1\mH^1(\partial^* S_1\cap B)+t_2\mH^1(\partial^* S_2\cap B)+t_3\mH^1(\partial^* S_3\cap B).
\]
\vskip.1in
\noindent
{\bf Problem 2.} Fix a number $\delta>0.$ Then for $h$ and $E_0$ as above, and disjoint sets $S_1,S_2$ and $S_3$ of finite perimeter in $B$ define
\[
m_0^\delta:=\inf \left\{E_0(S_1,S_2,S_3):\;\abs{B\setminus\cup S_\ell}\leq \delta,\;\partial S_\ell\cap \partial B=h^{-1}(p_\ell)\;\mbox{for}\;\ell=1,2,3\right\}.
\]

Regarding Problem 1, the regularity theory of minimizing planar partitions subject to volume constraints on each phase is developed, for instance, in \cite{Maggi,FM}. For our purposes, however, we require a version valid without volume constraints but subject to a Dirichlet condition, and which has additional properties specific to minimization within a ball. For this we quote the recent work in \cite{MNov}.

\bthm\label{Novack} (\cite{MNov}, Thm. 1.6)
If $(S_1^0,S_2^0,S_3^0)$ be a minimizer of Problem 1, then every connected component of $\partial S_\ell^0 \cap \partial S_m^0 \cap B$ is a line segment terminating at an interior triple junction $x\in \partial S_1^0 \cap \partial S_2^0 \cap \partial S_3^0 \cap B$, at $x\in \partial S_{\ell}^0 \cap \partial S_m^0 \cap \partial B$ for $\ell\neq m$ which is a point of discontinuity of $h$, or at a boundary triple junction $x\in \partial S_{1}^0 \cap \partial S_2^0 \cap S_3^0\cap \partial B$ which is a point of discontinuity of $h$. Moreover, there exists angles $\alpha_\ell$, $\ell=1,2,3$, satisfying \eqref{tj}
such that if $x\in B$ is an interior triple junction, for some 
$r_x>0$, $S_\ell^0\cap B_{r_x}$ for $\ell=1,2,3$ are circular sectors determined by $\alpha_\ell$.
Finally, every connected component $C$ of $S_\ell^0$ is convex and meets $\partial B$ along one or more arcs of $h^{-1}(p_\ell)$.
\ethm
 The last property, namely that every connected component $C$ of $S_\ell^0$ is convex and meets $\partial B$ along one or more arcs of $h^{-1}(p_\ell)$ is not stated in \cite{MNov}, Thm. 1.6, but is immediate since any island of phase could be filled in with a different phase, thereby lowering the total perimeter without disrupting the boundary condition.

Regarding Problem 1, we will also need the following corollary, which follows easily from Theorem \ref{Novack}.
\begin{cor}\label{tripbound}  Suppose $h\in BV(\partial B;P)$ has $k$ jump discontinuities for some non-negative integer $k$. Then there exists an integer $N(k)$ such that the total number of triple junctions appearing in any minimizer of Problem 1 cannot exceed $N(k)$.
\end{cor}
\begin{proof}
 Denote any minimizer of Problem 1 by $u_0$. Since any connected component of $\{u_0=p_\ell\}$ must meet $\partial B$ along one or more connected components of the set of boundary arcs $h^{-1}(p_\ell)$, and since the number of these arcs is necessarily bounded by $k$, we can conclude that for $\ell=1,2$ and $3$, the number of connected components of $\{u_0=p_\ell\}$ cannot exceed $k$ as well.
 Next we note that in light of the convexity of every component of $\{u_0=p_\ell\}$, any triple of components, one each from $\{u_0=p_\ell\}$ for $\ell=1,2,3$, can only meet at a triple junction at most once. Counting up all the possible triples, it follows that the number of triple junctions of a minimizer $u_0$ cannot exceed $k^3$.
\end{proof}

The regularity theory for Problem 2 is more subtle since minimizers will typically develop cusps to replace the triple junctions appearing in the solution of Problem 1, a phenomenon referred to in some literature as a ``wetting" of the singularities, see e.g. \cite{KBFM}. See Figure \ref{drywet}. Here we quote the following result: 
\bthm\label{Novack2} (\cite{MNov}, Thm. 1.4) 
Let $(S_1^\delta,S_2^\delta,S_3^\delta)$ be a minimizer of Problem 2, and denote by $G^\delta:=B\setminus \cup_{\ell=1}^3 S_\ell^\delta$. Then every connected component of $\partial S_\ell^\delta \cap \partial S_j^\delta \cap B$ is a line segment terminating either on $\partial B$ at a point of discontinuity of $h$ between $p_\ell$ and $p_j$ or at a point in $\partial S_\ell^\delta \cap \partial S_j^\delta \cap \partial G^\delta \cap  \overline{B}$. Referring to those points in $\partial S_\ell^\delta \cap \partial S_m^\delta \cap \partial G^\delta \cap  B$ and $\partial S_\ell^\delta \cap \partial S_m^\delta \cap \partial G^\delta \cap  \partial B$ as cusp and corner points, respectively, there exist positive $\kappa_\ell^\delta$ for $\ell=1,2,3$ such that
\begin{align}\label{curvature condition}
    t_1 \kappa_1^\delta=t_2 \kappa_2^\delta=t_3 \kappa_3^\delta 
\end{align}
and, for $\ell=1,2,3$, $\partial S_\ell^\delta \cap \partial G^\delta$ consists of a union of circular arcs of curvature $\kappa_\ell^\delta$, each of whose two endpoints are either a cusp point in $B$ or a corner point in $\partial B$ at a point of discontinuity of $h$. Furthermore, at cusp points, $\partial S_\ell^\delta\cap \partial G^\delta $ and $\partial S_m^\delta\cap \partial G^\delta $ meet $\partial S_\ell^\delta \cap \partial S_m^\delta$ tangentially. Finally, any connected component $C$ of $S_\ell^\delta$ is convex.
\ethm
 Since any admissible partition of $B$ for Problem 1 is also  admissible for Problem 2, it is obvious that $m_0^\delta\leq m_0.$ However, an  inequality in the reverse direction also holds.
\bthm
    \label{almostpartitions}
For any positive integer $k$, let  $h$ be any function in $ BV(\partial B;P)$ having no more than $k$ discontinuities. Then for any $\delta>0$, the infimum $m_0$ for Problem 1 and the infimum $m_0^\delta$ for Problem 2 are related via
\beq
m_0^\delta\geq m_0-\gamma(k)\,\delta^{1/2},
\label{keylb}
\eeq
for some constant $\gamma(k)$. 
\ethm
\begin{figure}[H]
	\centering
	\includegraphics[width = 0.6\textwidth]{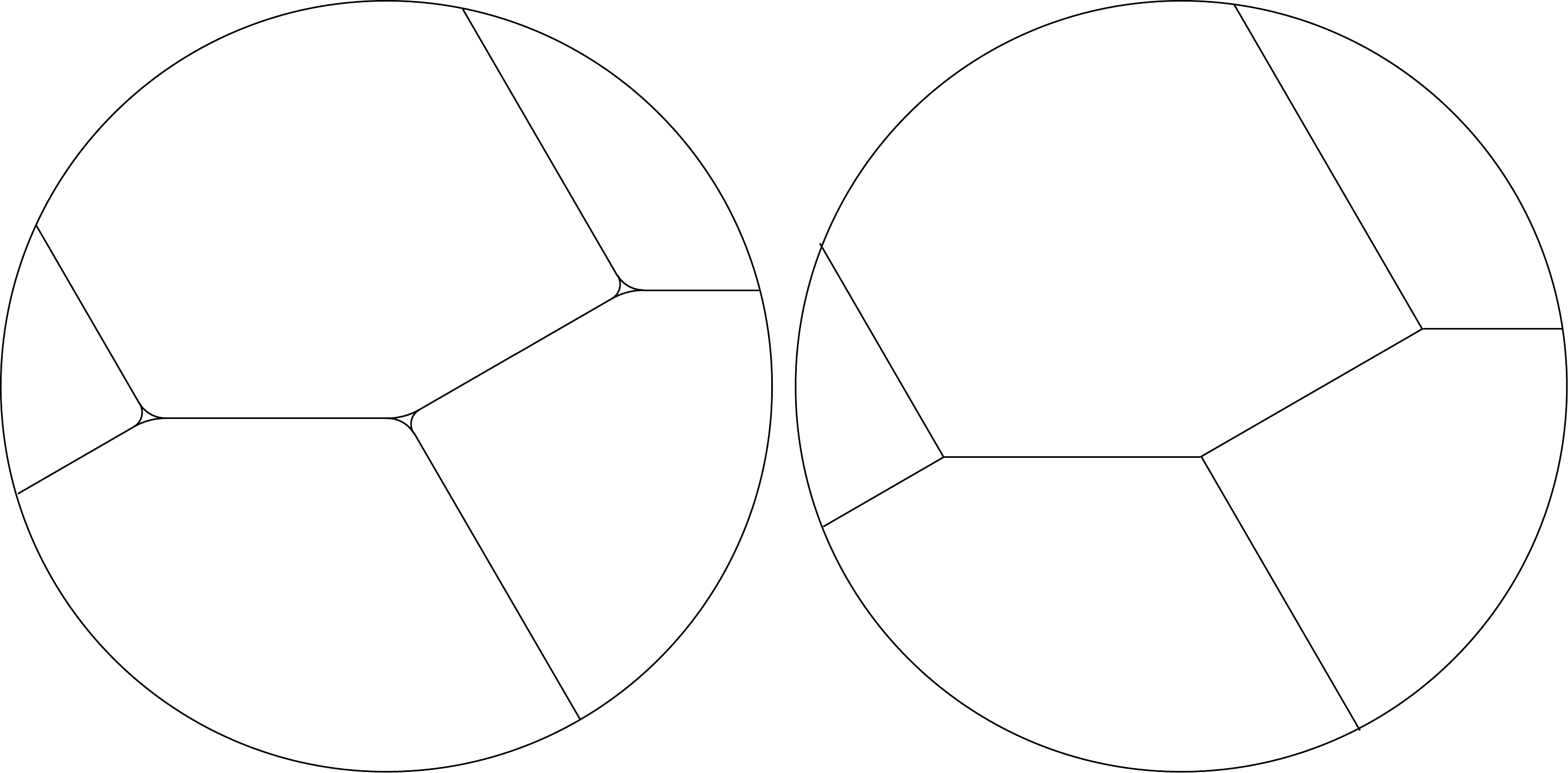}
	\caption{Minimizers of $E_0$ subject to the same Dirichlet condition for Problem 1 (right) and Problem 2 (left).}
	\label{drywet}
\end{figure}

\noindent
{\bf Proof of Theorem \ref{almostpartitions}}. We will view minimization of $E_0$ subject to the given Dirichlet condition on $\partial B$ and subject to the constraint \eqref{deltapart} as a problem of coloring most of $B$ with three colors, where for a minimizer
$\{S_1^\delta,S_2^\delta,S_3^\delta\}$ we rename $S_1^\delta$ as the yellow set $Y^\delta$,  $S_2^\delta$ as the red set $R^\delta$ and $S_3^\delta$ as the blue set $B^\delta$. 
Then the region in $B$ not covered by these sets, namely $G^\delta$, will be referred to as the gray set. Referring back to the formulation \eqref{tform} of the partitioning energy $E_0$, let us now write it as
\beq
E_0(Y^\delta,R^\delta,B^\delta)=c_Y\mH^1(\partial Y^\delta\cap B)+c_R\mH^1(\partial R^\delta\cap B)+c_B\mH^1(\partial B^\delta\cap B),\label{tforma}
\eeq
where we have changed notation to let $c_Y=t_1=$ the cost of `yellow boundary', $c_R=t_2$ and $c_B=t_3$.

Having fixed the boundary data $h\in BV(\partial B;P)$, we have $\partial Y^\delta\cap\partial B=h^{-1}(p_1),\,\partial R^\delta\cap\partial B=h^{-1}(p_2)$ and
$\partial B^\delta\cap\partial B=h^{-1}(p_3)$, and so $\partial B$ is partitioned into a finite number circular arcs, some yellow, some red and some blue, though we make no assumption that necessarily all three colors are present in the boundary data. We recall that we are assuming the total number of these arcs does not exceed $k$. It also follows from the Dirichlet condition that $\partial G^\delta$ meets $\partial B$ only at at most $k$ isolated points, if at all.

The main step in the proof consists of arguing that one can bound the number of components of $G^\delta$ by a constant depending only on $k$. Once this is established, the bound \eqref{keylb} will follow rather easily.

At this point, we will assume, with no loss of generality, that 
\beq
c_Y\leq \min\{c_R,c_B\}.\label{yellow}
\eeq

We now proceed with the proof in four steps. 
\vskip.1in\noindent
{\bf 1.} We first claim that with no loss of generality we may assume every component of $R^\delta$ and every component of $B^\delta$ must meet $\partial B$. That is, there are no islands of red or blue in the interior of $B$. 
This follows since any such island could be changed to yellow, either resulting in a new minimizer in the case of equality in \eqref{yellow}, or else contradicting the minimality of $\{Y^\delta,R^\delta,B^\delta\}$ in the case of strict inequality in \eqref{yellow}. As a consequence, the number of connected components of $R^\delta$ and of $B^\delta$ cannot exceed the number of red and blue boundary components dictated by $h$. In particular, both numbers are bounded by $k$. For the remainder of the argument, we will denote these components via
\beq
R^\delta=R^\delta_1\cup R^\delta_2\cup\ldots\cup R^\delta_{k_1}\quad\mbox{and}\quad
B^\delta=B^\delta_1\cup B^\delta_2\cup\ldots\cup B^\delta_{k_2}\label{redblue}
\eeq
for some integers $k_1=k_1(\delta)$ and $k_2=k_2(\delta)$ such that $k_1+k_2\leq k$.
\vskip.1in\noindent
{\bf 2.} Next we claim that we may assume every component of $G^\delta$ is simply connected. This follows since by Step 1, any non-simply connected component of $G^\delta$ would have one or more components of $Y^\delta$ consisting of full disks lying in its interior. We observe as a consequence of Proposition \ref{Novack2} that the outer boundary component of any component of $G^\delta$  
consists of a union of circular arcs of curvature $\kappa^\delta_Y, \kappa^\delta_R$ or $\kappa^\delta_B$ bowing into $G^\delta$, all meeting tangentially at cusp points, where we have renamed $\kappa^\delta_1$ as $\kappa^\delta_Y$, etc. Then we may shift any interior yellow disk until it touches this outer boundary component at two points without changing the total value of $E_0$, that is, creating a new minimizer. The only obstruction to sliding such an interior yellow disk over to the boundary would be that it first hits another yellow disk, but clearly two yellow disks, tangent at a point, is a non-minimizing configuration and so could not occur. In this manner, any minimizing configuration possessing a non-simply connected component of $G^\delta$ could be replaced by another minimizer having more components of $G^\delta$ than the original, but for which every component of the new $G^\delta$ is simply connected.

\vskip.1in\noindent
{\bf 3.} Our next goal is to bound the number of components of $Y^\delta$ which touch $G^\delta$, in the sense that their boundaries have nonempty intersection. It suffices to bound the number of those components which do not touch the boundary $\partial B$ since there are at most $k$ components which touch $\partial B$.

 Let then $Y^\delta_\ell$ be a yellow component which touches $G^\delta$. As described in Theorem \ref{Novack2}, the boundary of $Y^\delta_\ell$ is $C^1$ and consists of circular arcs and segments separated by points which are cusp singularities of the partition. Moreover, we may assume that  there are at least two such cusp points, for otherwise $\partial Y^\delta_\ell$ minus at most a point would be in a gray component and by sliding $Y^\delta_\ell$ in this gray component, we would obtain a minimizing partition where $Y^\delta_\ell$ has at least two cusp points on its boundary.
 
Let then $p_1$ and $p_2$ be cusp points on $\partial Y_\ell^\delta$, separated by a circular arc $\gamma\subset\partial Y^\delta_\ell\cap G^\delta$ of radius $r = 1/\kappa^\delta_Y$. These points also belong to the boundaries of red or blue components drawn from the collection \eqref{redblue}; call these two components $A_1$ and $A_2$. We claim that given components $A_1$ and $A_2$, each either red or blue, there can be at most two yellow components separated from the gray area by a boundary arc $\gamma$ whose endpoints belong to  $\partial A_1$ and $\partial A_2$, respectively.

We will argue this by first noting that the completion of any such circular arc $\gamma$ into a full circle yields a circle of radius $r = 1/\kappa^\delta_Y$. By Theorem \ref{Novack2}, this circle necessarily meets both $\partial A_1$ and $\partial A_2$ tangentially. We will thus rule out the possibility of there existing three or more such arcs $\gamma$ by showing that there can never exist three balls of the same radius, all exterior to $A_1$ and $A_2$, with all three meeting both $\partial A_1$ and $\partial A_2$ tangentially, unless all three balls have collinear centers. This is an elementary property of convex sets, but not being aware of a reference, we provide a proof. 

We note that necessarily the center of any such ball must be equidistant from $\partial A_1$ and $\partial A_2$. Therefore, for each $t>0$, we consider the possible intersections of the curves $\Gamma_1^t:=\{x:{\rm dist}(x,\partial A_1)=t\}$
and $\Gamma_2^t\:=\{x:{\rm dist}(x,\partial A_2)=t\}$. Since $A_1$
 and $A_2$ are two disjoint, convex sets that have $C^1$ boundaries within $B$, it follows that for all $t>0$, $\Gamma_1^t$ and $\Gamma_2^t$ are convex, closed curves that are also $C^1$ within $B$, and furthermore, $\Gamma_1^t\cap\Gamma_2^t$ are disjoint for $t$ small.

 Now consider the first time $t_1>0$ when $\Gamma_1^{t_1}$ meets $\Gamma_2^{t_1}$. This could happen along a line segment, since we recall that neither curve is necessarily strictly convex. However, for $\Gamma_1^{t_1}\cap\Gamma_2^{t_1}$ to consist of a line segment would mean that $\partial A_1$ and $\partial A_2$ must have boundary components that are parallel line segments. In this case, of course there exists a one-parameter family of circles with this tangency property, but necessarily, their centers all lie on the line segment $\Gamma_1^{t_1}\cap\Gamma_2^{t_1}$; that is, they are collinear.
 
 The other possibility is that at time $t_1$, the intersection $\Gamma_1^{t_1}\cap\Gamma_2^{t_1}$ consists of one point. Then for $0<t-t_1\ll 1$, the convexity of both curves means that the intersection will consist of two points, which in the context of our yellow components, allows for the possibility that two curves representing boundary components of two distinct elements of $Y^\delta$, say, $\partial Y^\delta_{\ell_1}\cap G^\delta$ and $\partial Y^\delta_{\ell_2}\cap G^\delta$, both meet $\partial A_1$ and $\partial A_2.$ This could certainly happen. See Figure \ref{twoyellow}.

 \begin{figure}[H]
	\centering
	\includegraphics[width = 0.6\textwidth]{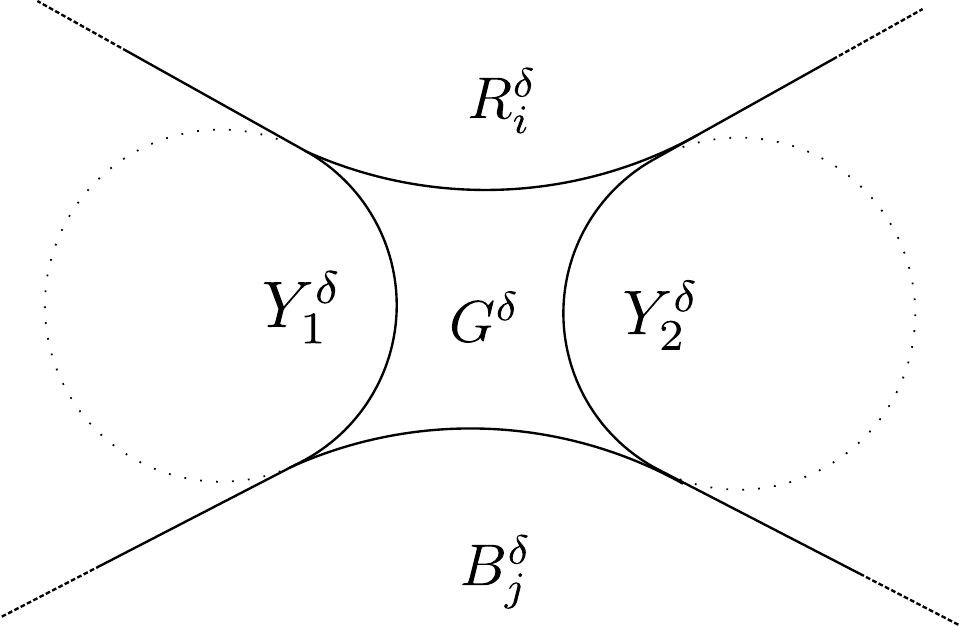}
	\caption{A configuration with two yellow components, both having a circular boundary arc whose endpoints meet the same pair of components from the collection in \eqref{redblue}, resulting in 4 cusp points.}
	\label{twoyellow}
\end{figure}

 However, we claim that as $t$ increases there cannot be more than two intersections of the two curves. We argue by contradiction and suppose that there exists a first time $t_2>t_1$ where $\Gamma_1^{t_2}\cap\Gamma_2^{t_2}$ consists of three points. Two of these points represent the continuous evolution of the original two points that emerged as $t$ passed through $t=t_1$ but at the third point, say $x\in \Gamma_1^{t_2}\cap\Gamma_2^{t_2}$, it must be the case that $\Gamma_1^{t_2}$ and $\Gamma_2^{t_2}$ meet tangentially, this being the first time a third point of intersection emerges. Denoting by $\mathcal{L}$ the line of tangency, it follows from convexity that either both $\{x:{\rm dist}(x,\partial A_1)<t_2\}$ and $\{x:{\rm dist}(x,\partial A_2)<t_2\}$ lie on the same side of $\mathcal{L}$ or they lie on opposite sides. But if they both lie on the same side then tracing back from $x$ a distance $t_2$ along the common inner normal to $\Gamma_1^{t_2}$ and $\Gamma_2^{t_2}$, one would arrive at a point in common to $\partial A_1$ and $\partial A_2$, which is impossible given that they are disjoint. If instead one supposes the two sets lie on opposite sides of $\mathcal{L}$, then that all earlier times, it must have been that
 $\Gamma_1^{t}\cap\Gamma_2^{t}$ was empty, contradicting the fact that the two curves met at the earlier time $t_1$.

Returning to the possibility of boundary components of $\partial A_1$ and $\partial A_2$ consisting of two parallel line segments, we note that any line segment on the boundary of a component corresponds to one and only one common boundary with a component of a different color, so in this context, it would correspond to only one yellow component meeting $A_1$, $A_2$ and $G^\delta$. 
This proves the claim that given components $A_1$ and $A_2$, each either red or blue, there can be at most two yellow components separated from the gray area by a boundary arc $\gamma$ whose endpoints belong to  $\partial A_1$ and $\partial A_2$, respectively.
 This in turn proves that the number of yellow components touching $G^\delta$ but not touching $\partial B$ is at most twice the number of pairs of components of red or blue, chosen from \eqref{redblue}, which is bounded by $2{k\choose 2}$.

\vskip.1in\noindent
{\bf 4.}  Now we turn to the task of bounding the number of components of $G^\delta$.  We will accomplish this by bounding the total number of cusp points. Any component of $G^\delta$ has boundary 
consisting of a union of circular arcs of curvature $\kappa^\delta_Y, \kappa^\delta_R$ or $\kappa^\delta_B$ bowing into $G^\delta$, all meeting tangentially at cusp points. 
 Each of these  circular boundary arcs must be a portion of boundary drawn from of the collection of sets
 \beq
 R^\delta_i\;\mbox{for}\; i\in\{1,2,\ldots,k_1\},\quad  B^\delta_j\;\mbox{for}\;j\in\{1,2,\ldots,k_2\},\;\mbox{or}\;  Y^\delta_\ell\;\mbox{for}\; \ell\in\{1,2,\ldots,k_3\},\label{allsets}
 \eeq 
 with $k_1+k_2+k_3\leq k+2{k\choose 2}$ in light of (1) and (3) above.
 
  Now consider any yellow/red, yellow/blue or red/blue pair taken from this collection, say for instance $R^\delta_1$ and $B^\delta_1$, and suppose $\partial R^\delta_1$ and $\partial B^\delta_1$ each have circular arcs bordering the same component of $G^\delta$ such that these two arcs meet at a particular cusp point. By the convexity of red and blue components, $\partial R^\delta_1\cap \partial B^\delta_1$ must consist solely of a line segment one of whose endpoints is this cusp point. It follows that $\partial R^\delta_1$ and $\partial B^\delta_1$ can both meet the boundary of some other gray component at at most one other cusp point, namely at a cusp point sitting at the other endpoint of their one common boundary segment. The same argument could be made between any pairing of a $Y^\delta_\ell$ with any $R^\delta_i$ or $B^\delta_j$, provided $Y^\delta_\ell$ is not a full disk. On the other hand, if $Y^\delta_\ell$ is a full disk then its intersection with any $R^\delta_i$ or $B^\delta_j$ results in only one cusp point due to the convexity of both sets involved. Estimating crudely,  the total number of yellow/red, yellow/blue or red/blue pairs drawn from the collection \eqref{allsets} is bounded by
 \[{k+2{k\choose 2}\choose 2}.
 \]
  Hence, as just argued, the total number of cusp points in a minimizing configuration $\{Y^\delta,R^\delta,B^\delta\}$ cannot exceed twice this number. But since a closed curve comprised of concave circular arcs requires at least three such arcs, it follows that any component of $\partial G^\delta$ must have at least three cusp points. Thus we can bound the total number of gray boundary components and hence, the total number of gray components, by 
  \beq
  C(k):=\frac{2}{3}{k+2{k\choose 2}\choose 2}.\label{cusptotal}
  \eeq
\vskip.1in\noindent
{\bf 5.} Finally, we are ready to establish inequality \eqref{keylb}. To this end, we now build out of $\{Y^\delta,R^\delta,B^\delta\}$ a competitor in the problem, denoted by Problem 1 at the outset of this section, of minimizing $E_0$ among full partitions of the disk $B$, subject to the Dirichlet condition $h$, by defining
$\tilde{Y}^\delta:=Y^\delta\cup G^\delta$. Then $\{ \tilde{Y}^\delta,R^\delta,B^\delta\}$ competes with the minimizer of this problem, denoted in the statement of Theorem \ref{almostpartitions} by $\{S^0_1,S^0_2,S^0_3\}$, and so we have
\beq
E_0(\tilde{Y}^\delta,R^\delta,B^\delta)\geq E_0(S^0_1,S^0_2,S^0_3),\label{a1}
\eeq
as well as
\beq
 E_0(\tilde{Y}^\delta,R^\delta,B^\delta) \leq E_0(Y^\delta,R^\delta,B^\delta)+c_Y\mH^1(\partial G^\delta).
 \label{a2}
\eeq
Now in light of \eqref{curvature condition} and our bound on the total number of cusp points, any boundary of a component of $G^\delta$ consists of at most $2\,{k+2{k\choose 2}\choose 2}$ circular arcs of radius at most
\[
\max{\left\{\frac{1}{\kappa_R},\frac{1}{\kappa_B}\right\}}=\frac{\max{\{c_R,c_B\}}}{c_Y}\frac{1}{\kappa^\delta_Y}.
\]
Consequently, bounding the length of any arc by the perimeter of the corresponding full circle, we can assert that
\beq
\mH^1(\partial G^\delta)\leq 2\pi C(k)\frac{\max{\{c_R,c_B\}}}{c_Y}\frac{1}{\kappa^\delta_Y}.
\label{a3}
\eeq
Then turning from bounding perimeter to bounding area, we note that the area enclosed by the boundary of any gray component in this minimizing configuration is bounded from below by the area of the smallest possible (simply connected) gray component, namely the one formed by just three arcs arising from the tangential contact of one yellow, one red and one blue arc. This number could of course be computed precisely but for our purposes, it suffices to observe through another appeal to \eqref{curvature condition} that it is 
given by $\alpha \frac{1}{(\kappa^\delta_Y)^2}$ for some positive constant $\alpha=\alpha(c_R,c_B)$, where we have expressed this minimal area in terms of $\kappa^\delta_Y$ though of course, we could have expressed it in terms of either of the other two curvatures as well. Hence, assuming there exists at least one component of $G^\delta$, we have the following estimate on the curvature $\kappa^\delta_Y$:
\beq
\alpha\, \frac{1}{(\kappa^\delta_Y)^2}\leq\abs{G^\delta}\leq \delta.\label{a4}
\eeq
Combining \eqref{a1}-\eqref{a4}, we conclude that
\[
E_0(Y^\delta,R^\delta,B^\delta)\geq E_0(S^0_1,S^0_2,S^0_3)-\gamma(k)\,\delta^{1/2}\quad
\mbox{with}\quad\gamma(k):=\frac{2\pi C(k)}{\sqrt{\alpha}}\max\{c_R,c_B\}.
\]
\qed
\vskip.2in\noindent
{\it Acknowledgments.} P.S. would like to thank Mihai Ciucu, Francesco Maggi and Frank Morgan for helpful suggestions. The research of P.S. was supported by a Simons Collaboration grant 585520, an NSF grant DMS 2106516 and through a visit to Universit\'e Paris-Est Cr\'eteil made possible through the Bezout Chair program. He is grateful for the hospitality shown during this visit.

\bibliographystyle{acm}
\bibliography{EP}

\end{document}